\documentclass{amsart}

\usepackage{amssymb}
\usepackage{amsthm}
\usepackage{amsmath}
\usepackage{latexsym}
\usepackage{mathrsfs}
\usepackage{amsrefs}

\newtheorem{theorem}{\bf Theorem}

\newtheorem{lemma}{\bf Lemma}[section]

\newtheorem{proposition}{\bf Proposition}[section]

\newcommand{\R}{{\mathbb R}}
\newcommand{\rn}{{\mathbb R}^n}
\newcommand{\N}{{\mathbb N}}

\newcommand{\lam} {\lambda}
\newcommand{\lap}{\Delta}

\newcommand {\crit} {2_k^\sharp}

\begin{document}

\title[Polyharmonic Operators on  a Compact Riemannian Manifold]{GJMS-type Operators on   a  compact Riemannian manifold: Best constants and Coron-type solutions}
\author{Saikat Mazumdar}
\address{Institut Elie Cartan de Lorraine, Universit\'e de Lorraine, BP 70239, 54506    Vand\oe uvre-l\`es-Nancy, France}
\email{saikat.mazumdar@univ-lorraine.fr}
\date{December 7th, 2015, revised July 5th, 2016}
\thanks{This work is part of the PhD thesis of the author, funded by "F\'ed\'eration Charles Hermite" (FR3198 du CNRS) and "R\'egion Lorraine". The author acknowledges these two institutions for their supports.}
\maketitle

\begin{abstract}
In this paper we investigate the existence of solutions to a nonlinear  elliptic problem involving critical Sobolev exponent for  a polyharmonic operator   on a Riemannian manifold $M$. We first  show  that  the best  constant of the  Sobolev embedding on a  manifold can be chosen as close as one wants to the Euclidean one, and as a consequence derive the existence of minimizers  when the energy functional goes below a quantified threshold. Next,  higher energy solutions are obtained  by Coron's topological method, provided that the minimizing solution does not exist.  To perform this topological argument, we  overcome the difficulty of dealing with  polyharmonic operators on a Riemannian manifold and adapting Lions's concentration-compactness lemma. Unlike Coron's original argument for a bounded domain in $\R^{n}$, we need to do more than chopping out a small ball from the manifold $M$.  Indeed, our topological assumption  that a small sphere on $M$ centred at a point $p \in  M$  does not retract to  a point in $M\backslash \{ p \}$  is necessary, as shown for  the case of the  canonical sphere where chopping out a small  ball is not enough.
\end{abstract}


\section{Introduction}
Let $M$ be a compact manifold of dimension $n\geq 3$ without boundary. Let $k$ be a positive integer such that $2k < {n}$. Taking inspiration from the construction of the ambient metric of Fefferman-Graham \cite{fg1} (see \cite{fg2} for an extended analysis of the ambient metric), Graham-Jenne-Mason-Sparling \cite{gjms} have defined a family of conformally invariant operators defined for any Riemannian metric. More precisely, for any Riemannian metric $g$ on $M$, there exists a local differential operator $P_g: C^\infty(M)\to C^\infty(M)$ such that $P_g=\Delta_g^k+lot$ where $\Delta_g:=-\hbox{div}_g(\nabla)$, and, given $u\in C^\infty(M)$ and defining $\hat{g}=u^{\frac{4}{n-2k}}g$, we have that
\begin{equation}\label{eq:conf}
P_{\hat{g}}(\varphi)=u^{-\frac{n+2k}{n-2k}}P_g\left( u\varphi\right)\hbox{ for all }\varphi\in C^\infty(M).
\end{equation}

\smallskip\noindent Moreover, $P_g$ is self-adjoint with respect to the $L^2-$scalar product. A scalar invariant is associated to this operator, namely the $Q-$curvature, denoted as $Q_g\in C^\infty(M)$.  When $k=1$, $P_g$ is the conformal Laplacian and the $Q-$curvature is the scalar curvature multiplied by a constant. When $k=2$, $P_g$ is the Paneitz operator introduced in \cite{paneitz}. The $Q-$curvature was introduced by  Branson and \O rsted  \cite{bo}. The definition of $Q_g$ was then generalized by Branson \cites{br1,br2}. In the specific case $n>2k$, we have that $Q_g:=\frac{2}{n-2k}P_g(1)$. Then, taking $\varphi\equiv 1$ in \eqref{eq:conf}, we get that $P_gu= \frac{n-2k}{2}Q_{\hat{g}}u^{\frac{n+2k}{n-2k}}$ on $M$. Therefore, prescribing the $Q-$curvature in a conformal class amounts to solving a nonlinear elliptic partial differential equation(PDE )of $2k^{th}$ order. Results for  the prescription of  the $Q-$curvature problem  for the Paneitz operator (namely $k=2$) are in Djadli-Hebey-Ledoux \cite{DHL}, Robert \cite{robertpams}, Esposito-Robert \cite{espositorobert}. Recently, Gursky-Malchiodi \cite{gm} proved the existence of a metric with constant $Q-$curvature (still for $k=2$) provided certain geometric hypotheses on the manifold $(M,g)$ holds. These hypotheses have been simplified by Hang-Yang \cite{hy1} (see the lecture notes \cite{hy2})

\medskip\noindent In the present paper, we are interested in a generalization of the prescription of the $Q-$curvature problem. Namely, given $f\in C^\infty(M)$, we investigate the existence of $u\in C^\infty(M)$, $u>0$, such that
\begin{equation}\label{eq:P:f}
P u=f u^{\crit-1}\hbox{ in }M,
\end{equation}
where $\crit:=\frac{2n}{n-2k}$ and $P: C^\infty(M)\to C^\infty(M)$ is a  smooth self-adjoint $2k^{th}$ order partial differential operator  defined by
  \begin{align}{\label{the op}}
 Pu= \Delta_{g}^{k}u + \sum \limits_{l=0}^{k-1} (-1)^{l}\nabla^{j_{l} \ldots j_{1}} \left( A_{l}(g)_{i_{1} \ldots i_{l}, j_{1} \ldots j_{l}} \nabla^{i_{1} \ldots i_{l}} u\right)
 \end{align}
 where the indices are raised via the musical isomorphism and for all $l \in \{ 0, \ldots, k-1\}$, $A_{l}(g) $  is a smooth symmetric $T_{2l}^{0}$-tensor field on $M$ (that is: $A_{l}(g)(X,Y)  =A_{l}(g)(Y,X) $ for all $T_{0}^{l}$-tensors $X,Y$ on $M$). When $P:=P_g$, then \eqref{eq:P:f} is equivalent to say that $Q_{\hat{g}}=\frac{2}{n-2k}f$ with $\hat{g}=u^{\frac{4}{n-2k}}g$.

\medskip\noindent The conformal invariance \eqref{eq:conf} of the geometric operator $P_g$ yields obstruction to the existence of solutions to \eqref{eq:P:f}. The historical reference here is Kazdan-Warner \cite{kw}; for the general GJMS operators, we refer to Delano\"e-Robert \cite{dr}. In particular, it follows from \cite{dr} that on the canonical sphere $(\mathbb{S}^n,\hbox{can})$, there is no positive solution $u\in C^\infty(\mathbb{S}^n)$ to $P_{\hbox{can}}u=(1+\epsilon\varphi) u^{\crit-1}$ for all $\epsilon\neq 0$ and all first spherical harmonic $\varphi$. For the conformal Laplacian (that is $k=1$), Aubin \cite{aubin} proved that the existence of solutions is guaranteed if a functional goes below a specific threshold. We generalize this result for any $k\geq 1$ in Theorem \ref{th:min}. In the case of a smooth bounded domain, Coron \cite{coron} introduced a variational method based on topological arguments,  provided  the minimizing solution does not exist.   Our main theorem is in this spirit:

  \begin{theorem}{\label{main}} 
  Let $(M,g)$ be a  smooth, compact Riemannian manifold  of dimension $n$ and let $k$ be a positive integer such that $2k < {n}$.  We let $P$ be a coercive operator as in \eqref{the op}. Let  $\iota_{g}>0$ be the  injectivity radius of the manifold $M$.  Suppose that the manifold $M$ contains a point  $x_{0}$ such that  the embedded $(n-1)-$ dimensional sphere $\mathbb{S}_{x_0}({\iota_{g}}/ {2}):=\{x\in M/\, d_g(x,x_0)={\iota_{g}}/ {2}\}$ is not contractible in $M\backslash \{x_{0}\}$. Then there exists $\epsilon_{0}\in (0, \frac{\iota_{g}} {2})$ such that the equation
   \begin{eqnarray}\label{eq:coron}
 \left \{ \begin{array} {lc}
          Pu  = \left| u \right|^{ \crit - 2} u \qquad  \text{in} ~ \Omega_{M}\\
          D^{\alpha} u=0  \ \  \  \qquad \qquad\text{on } ~\partial \Omega_{M} \quad \text{for} \ \ |\alpha| \leq k-1
            \end{array} \right. 
\end{eqnarray} 
has a non-trivial $C^{2k}(\Omega_{M})$ solution for $\Omega_{M}:=M \backslash \overline{B}_{x_{0}}(\epsilon_{0})$. Moreover, if the Green's Kernel of $P$ on $\Omega_M$ is positive, then we can choose $u>0$.\end{theorem}
In the original result of Coron \cite{coron} (see also Weth and al. \cite{BWW} for the case $k=2$), the authors work with a smooth domain of $\R^n$ and assume that it has a small ``hole". In the context of a compact manifold, this assumption is not enough: indeed, the entire compact manifold minus a small hole might retract on a point. We discuss the example of the canonical sphere in Section \ref{sec:counterex}, where the existence of a hole is not sufficient to get solutions to \eqref{eq:P:f}.

\noindent
Concerning higher-order problems, we refer to Bartsch-Weth-Willem \cite{BWW}, Pucci-Serrin \cite{pucci-serrin}, Ge-Wei-Zhou \cite{gwz}, the general monograph Gazzola-Grunau-Sweers \cite{GGS} and the references therein.

\medskip\noindent Among other tools, the proof of Theorem \ref{main} uses a Lions-type Concentration Compactness Lemma adapted to the context of a Riemannian manifold: this will be the object of Theorem \ref{th:CC}. 

\medskip\noindent Equation \eqref{eq:P:f} has a variational structure. Since $P$ is self-adjoint in $L^2$, we have that for all $u,v \in C^{\infty}(M)$.
 \begin{align}\label{def:upv}
 \int \limits_{M}  u P(v)\, dv_{g}=  \int \limits_{M}  v P(u)\, dv_{g}=  \int \limits_{M}  \Delta_g^{k/2} u \Delta_g^{k/2} v \,{dv_g}  + \sum \limits_{l =0}^{k-1}  \int \limits_{M} A_{l}(g)(\nabla^{l} u,\nabla^{l} v) \,dv_{g}
 \end{align}
where
$$\Delta_{g}^{l/2} u:=\left\{\begin{array}{ll}
\Delta_g^m u&\hbox{ if }l=2m\hbox{ is even }\\
\nabla \Delta_g^m u&\hbox{ if }l=2m+1\hbox{ is odd }
\end{array}\right.$$
and, when $l=2m+1$ is odd, $\Delta_g^{k/2} u\,  \Delta_g^{k/2} v=\left(\nabla \Delta_g^m u,\nabla \Delta_g^m v\right)_g$. If $P$ is coercive and $f>0$, then, up to multiplying by a constant, any solution $u\in C^\infty(M)$ to \eqref{eq:P:f} is a critical point of the functional
 \begin{align}{\label{functional}}
u\mapsto \displaystyle{ J_P(u) := \frac{ \int \limits_{{M}} u P(u)\, dv_{g}  }{\left( ~ \int \limits_{{M}} f |u|^{  \crit}~ dv_{g} \right)^{2/  \crit}} }.
\end{align}
It follows from \eqref{def:upv} that $J_P$ makes sense in the Sobolev spaces $H_{k}^2({M})$, where for $1\leq l\leq k$, $H_l^2(M)$ which is the completion of $C^{\infty}(M)$ with respect to the $u\mapsto  \sum_{\alpha =0}^{l} \Vert \nabla^{\alpha} u\Vert_2$. Equivalently (see Robert \cite{robertgjms}), $H_{l}^2({M})$ is also  the completion of the space $C^{\infty}(M)$ with respect to the norm
\begin{equation}
 \left\| u\right\|_{H_{l}^2}^2 := 
 \sum \limits_{\alpha =0}^{l} \int \limits_{M}   (\Delta_g^{{\alpha}/2} u )^2 ~{dv_g}.
\end{equation}
By the Sobolev embedding theorem we get a continuous but not compact embedding of $H_{k}^{2}(M)$ into $L^{\crit}(M)$. The continuity of the embedding $H_{k}^{2}(M)\hookrightarrow L^{\crit}(M)$ yields a pair of real numbers $A,B $ such that for all $u\in H_{k}^{2}(M)$ 
\begin{align}{\label{inequality}}
\left\|u \right\|_{L^{\crit}}^{2}  \leq A  \int \limits_{M}   (\Delta_g^{{k}/2} u )^2 \, dv_g  + B \left\| u\right\|_{H_{k-1}^{2}}^{2}.
\end{align}
See for example Aubin \cite{aubinbook} or Hebey \cite{hebeycims}. Following the terminology introduced by Hebey, we then define
\begin{align}
\mathcal{A}(M):= \inf \{ A \in \R : \exists ~B \in  \R \ \text{with the property that inequality $(\ref{inequality})$ holds}\}.
\end{align}
As for the classical case $k=1$ (see Aubin \cite{aubinbook}), the value of ${\mathcal A}(M)$ depends only on $k$ and the dimension $n$. More precisely, we let $\mathscr{D}^{k,2}(\R^n)$ be the completion of $C^\infty_c(\R^n)$ for the norm $u\mapsto \Vert \Delta^{{k}/2} u\Vert_{2}$, and we define $K_0(n,k)>0$ 
\begin{equation}\label{sinq}
\frac{1}{K_0(n,k)} := \inf \limits_{u \in \mathscr{D}^{k,2}(\R^n)\backslash \{0\}} \frac{\int_{\R^n} ( \Delta^{k/2 }u )^2 \, dx}
 {\left(\int_{\R^n}|u|^{\crit}\, dx\right)^{\frac{2}{\crit}}}.
 \end{equation}
as the best constant in the Sobolev's continuous embedding $\mathscr{D}^{k,2}(\R^n) \hookrightarrow L^{\crit}(\R^n)$. Our second result is the following:
\begin{theorem}{\label{optimal inequality}}
Let $(M,g)$  be a  smooth, compact Riemannian manifold of dimension $n$ and  let $k$ be a positive integer such that $2k < {n}$. Then ${\mathcal A}(M)=K_0(n,k)>0$. In particular, for any $\epsilon>0$, there exists $B_{\epsilon} \in \R$ such that for all $u \in H_{k}^{2}(M)$ one has
\begin{align}\label{eq:eps}
\left(\int_M |u|^{\crit}\, dv_g\right)^{\frac{2}{\crit}} \leq (K_{0}(n,k) + \epsilon)  \int \limits_{M}   (\Delta_g^{{k}/2} u )^2 ~{dv_g}  + B _{\epsilon} \left\| u\right\|_{H_{k-1}^{2}}^{2}.
\end{align}
\end{theorem}
As a consequence of this result, we will be able to prove the existence of solutions to \eqref{eq:P:f} when the functional $J_P$ goes below a quantified threshold, see Theorem \ref{th:min}.

\medskip\noindent This paper is organized as follows. In Section \ref{sec:best}, we study the best-constant problem and prove Theorem \ref{optimal inequality}. In Section \ref{sec:min}, we prove Theorem \ref{th:min} by classical minimizing method. In Section \ref{CC}, we prove a Concentration-Compactness Lemma in the spirit of Lions. Section \ref{sec:coron} is devoted to test-functions estimates and the proof of the existence of solutions to \eqref{eq:coron} via a Coron-type topological method. Section \ref{sec:pos} deals with positive solutions,  and Section \ref{sec:counterex} with the necessity of the topological assumption of Theorem \ref{main}. The appendices concern regularity and a general comparison between geometric norms. 

\subsection*{Acknowledgements}

I would like to express my deep gratitude to Professor Fr\'ed\'eric Robert  and Professor Dong Ye, my thesis  supervisors, for their patient guidance, enthusiastic encouragement and useful critiques of this work.





\section{The Best Constant}\label{sec:best}

\noindent
It follows from Lions {\cite{PLL}} and Swanson {\cite{swanson}} that the extremal functions for the Sobolev inequality \eqref{sinq} exist and are exactly multiples of the functions
\begin{align}
U_{a,\lambda}=  \alpha_{n,k}\left( \frac{\lambda}{1+ \lambda^{2}|x-a|^{2}} \right)^{\frac{n-2k}{2}}~ a \in \R^{n}, \lambda >0
\end{align}
where the choice of $\alpha_{n,k}$'s are such that for all $\lam$, $ \left\|  U_{a,\lambda} \right\|_{\crit}=1$ and   
$\left\| U_{a,\lambda} \right\|_{\mathscr{D}^{k,2}}^{2}=\frac{ 1 }{K_{0}(n,k)} $.
They satisfies  the equation $\Delta^{k} u = \frac{1}{K_{0}(n,k)}\left| u\right|^{2_{k}^{\sharp}-2} u$ in $\rn$

\medskip\noindent
Next we consider the case of a compact Riemmanian manifold. The first result we have in this direction is the following.
\begin{lemma}{\label{a b lemma}}
Let $(M,g)$  be a  smooth, compact Riemannian manifold of dimension $n$ and  let $k$ be a positve integer such that $2k < {n}$. Any constant $A$ in inequality $(\ref{inequality})$ has to be greater than or equal to $K_{0}(n,k)$, whatever the  constant $B$ be.
\end{lemma}
\noindent{\it Proof of Lemma \ref{a b lemma}:} We fix $\epsilon>0$ small. It follows from Lemma {\ref{normcomparison}} that there exists, $\delta_{0} \in( 0,\iota_{g})$ depending only on $(M,g), \epsilon$, where $\iota_{g}$ is the injectivity radius of $M$, such that for any point $p \in M$, any $0 < \delta< \delta_{0}$, $ l \leq k$  
 and $u \in C_c^\infty(B_{0}(\delta))$
   \begin{align}
 \int \limits_{M}  ( \Delta^{l/2}_{{g}} (u \circ exp_{p}^{-1}) )_{g}^2  ~{dv_{g}}
  \leq  {(1+\epsilon)} \int \limits_{\R^n} ( \Delta^{l/2} u )^2 ~{dx}  \qquad 
  \end{align}
  and
  \begin{align}
 (1-\epsilon) \left(~\int \limits_{\R^{n}} |u|^{2_{k}^{\sharp}} ~ {dx} \right)^{2/{2_{k}^{\sharp}}} \leq  \left( \int \limits_{M} |u \circ exp_{p}^{-1}|^{2_{k}^{\sharp}} ~ {dv_{g}} \right)^{2/{2_{k}^{\sharp}}}
 \end{align}
 \medskip
 
 \noindent
 Then plugging the above inequalities into \eqref{inequality} we obtain that any $u \in C_c^\infty\left( B_{0}(\delta)\right)$ satisfies 
 \begin{align}
\displaystyle  \left(\int \limits_{\R^{n}} |u|^{2_{k}^{\sharp}} ~ {dx} \right)^{2/{2_{k}^{\sharp}}} \leq   \frac{1+\epsilon}{ 1-\epsilon} A  \int \limits_{\R^{n}}   ( \Delta^{{k}/2} u )^2 ~{dx}  + C_\epsilon \sum_{l=0}^{k-1}\int_{\R^n}|\nabla^l u|^2\, dx\label{ineq:sobo:1}.
\end{align}
 \medskip

 \noindent
 Let $v \in C^\infty_c\left( \R^{n}\right) $ with $supp(v) \subset B_{0}(R_{0})$. For $\lambda >1$ let $v_{\lambda}= v(\lambda x )$. Then for $\lambda$ large, $supp(v_{\lambda}) \subset B_{0}(\delta) $. Taking $u\equiv v_\lambda$ in \eqref{ineq:sobo:1}, a change of variable yields

  \begin{align}
 \displaystyle   \frac{1}{\lambda^{n-2k}}\left(~\int \limits_{\R^{n}} |v|^{2_{k}^{\sharp}} ~ {dx} \right)^{2/{2_{k}^{\sharp}}} \leq \frac{1+\epsilon}{ 1-\epsilon} \cdot
   \frac{ A}{\lambda^{n-2k}}  \int \limits_{\R^{n}}   ( \Delta^{{k}/2} v )^2 ~{dx}  + C_\epsilon \sum_{l=0}^{k-1}\frac{1}{\lambda^{n-2l}}\int_{\R^n}|\nabla^l v|^2\, dx.
 \end{align}
 Multiplying by $\lambda^{n-2k}$ and letting $\lambda\to +\infty$, we get that for all $v \in \mathscr{D}^{k,2}(\R^{n})$, we have
  \begin{align}
 \displaystyle  \left(~\int \limits_{\R^{n}} |v|^{\crit} ~ {dx} \right)^{2/{\crit}} \leq 
   \frac{1+\epsilon}{ 1-\epsilon} A  \int \limits_{\R^{n}}   ( \Delta^{{k}/2} v )^2 ~{dx}.
 \end{align}
 Therefore $ \frac{1+\epsilon}{ 1-\epsilon} A \geq K_0(n,k)$ for all $\epsilon>0$, and letting $\epsilon\to 0$ yields $A\geq K_0(n,k)$. This ends the proof of Lemma \ref{a b lemma}.\hfill$\Box$
  
\noindent We now prove \eqref{eq:eps} to get Theorem \ref{optimal inequality}. 

\medskip\noindent{\bf Step 1: A local inequality.} From a  result of Anderson ({\emph{Main lemma $2.2$}} of  \cite{anderson})   it follows  that  for  any point  $p \in M $ there exists a harmonic  coordinate chart  $\varphi$ around $p$. Then from   Lemma {\ref{normcomparison}}, for any $ 0<  \epsilon <1 $, there exists $\tau  >0$ small  enough such that for any point $p \in M $ and  for any $u \in C_c^\infty\left( B_{p}(\tau)\right)$, one has
\begin{align}
 \int \limits_{\R^n} ( \Delta^{k/2} (u \circ \varphi^{-1}) )^2 ~{dx}  
  \leq  
 {\left(1+\frac{\epsilon}{3K_0(n,k) }\right)}  \int \limits_{M}  (\Delta^{k/2}_{{g}} u  )^2  ~{dv_{g}}
  \end{align}
  and
  \begin{align}
  \left(~\int \limits_{M} |u|^{2_{k}^{\sharp}} ~ {dv_{g}} \right)^{2/{2_{k}^{\sharp}}} \leq   {\left(1+\frac{\epsilon}{3K_0(n,k) }\right)} \left( ~\int \limits_{\R^{n}} |u \circ \varphi^{-1}|^{2_{k}^{\sharp}} ~ {dx} \right)^{2/{2_{k}^{\sharp}}}.
 \end{align}
The expression for the Laplacian $\Delta_{g}$ in the harmonic coordinates is $\Delta_{g}u= - g_{ij} \partial_{ij} u$. 


\noindent
Then \eqref{sinq} implies that for any $u \in C_c^\infty\left( B_{p}(\tau)\right)$

\noindent
\begin{align}{\label{loc}}
 \left( ~\int \limits_{M} |u |^{2_{k}^{\sharp}} ~ {dv_{g}} \right)^{2/{2_{k}^{\sharp}}}  \leq 
 (K_0(n,k)  +\epsilon)   \int \limits_{M}  ( \Delta^{k/2}_{{g}} u  )^2  ~{dv_{g}}.
\end{align}
\medskip

\medskip\noindent{\bf Step 2: Finite covering and proof of the global inequality.} Since $M$ is compact, it can be covered by a finite number of  balls $B_{p_{i}}(\tau/2)$, $i=1,\ldots,N$.
Let $ \alpha_{i} \in C^\infty_c(B_{p_{i}}(\tau))$ be such that $0 \leq \alpha_{i} \leq 1$ and $\alpha_{i}=1$ in $B_{p_{i}}(\tau/2)$. We set

\begin{align}
\displaystyle{\eta_{i}= \frac{\alpha_{i}^{2}}{ \sum \limits_{j=1}^{N} \alpha_{j}^{2}}   }.
\end{align}
Then $\left( \eta_{i} \right)_{{i=1,\ldots,N}}$ is a partition of unity subordinate to the cover $\left( B_{p_{i}}(\tau)\right)_{{i=1,\ldots,N}}$ such that $\sqrt{\eta_{i}}$'s are smooth and $\sum \limits_{i=1}^{N} \eta_{i} = 1$. In the sequel, $C$ denote any positive constant depending on $k,n$, the metric $g$ on $M$ and  the functions
$\left(  \eta_{i}\right)_{i=1,\ldots,N}$. Now for any $u \in C^{\infty}(M)$, we have
\begin{align}
\|u \|_{{2_{k}^{\sharp}}}^{2} = \|u^{2} \|_{{{2_{k}^{\sharp}}/2} } =  \left\| \sum \limits_{i=1}^{N} \eta_{i} u^{2} \right \|_{{{2_{k}^{\sharp}}/2} }  \leq 
\sum \limits_{i=1}^{N} \left\| \eta_{i} u^{2} \right\|_{{{2_{k}^{\sharp}}/2} } =  \sum \limits_{i=1}^{N} \left\| \sqrt{\eta_{i}} u\right\|_{{2_{k}^{\sharp}}}^{2}.
\end{align}
\medskip

\noindent
So for any $u \in C^{\infty}(M)$, using inequality $( \ref{loc} ) $ we obtain that
\begin{align}{\label{semi-local}}
 \left( ~\int \limits_{M} |u |^{2_{k}^{\sharp}} ~ {dv_{g}} \right)^{2/{2_{k}^{\sharp}}}  \leq 
(K_0(n,k)  +\epsilon)\sum \limits_{i=1}^{N}  \int \limits_{M}  ( \Delta^{k/2}_{{g}} (\sqrt{\eta_{i}} u) )_{g}^2  ~{dv_{g}}.
\end{align}
\medskip\noindent
Next we claim that there exists $C>0$ such that
\begin{align}{\label{main claim}}
\sum \limits_{i=1}^{N}  \int \limits_{M}  ( \Delta^{k/2}_{{g}} (\sqrt{\eta_{i}} u) )^2  ~{dv_{g}} \leq
  \int \limits_{M}  (\Delta^{k/2}_{{g}}u )^2  ~{dv_{g}} + C \left\| u\right\|_{H_{k-1}^{2}}^{2}.
\end{align}
\medskip

\noindent
Assuming that $(\ref{main claim})$ holds we have from $( \ref{semi-local})$
\begin{align}
 \left( ~\int \limits_{M} |u |^{2_{k}^{\sharp}} ~ {dv_{g}} \right)^{2/{2_{k}^{\sharp}}}  \leq 
(K_0(n,k)  +\epsilon)  \int \limits_{M}  (\Delta^{k/2}_{{g}}u )_{g}^2  ~{dv_{g}}
 + (K_0(n,k)  +\epsilon)   C \left\| u\right\|_{H_{k-1}^{2}}^{2}.
\end{align}
This proves \eqref{eq:eps}, and therefore, with Lemma \ref{a b lemma}, this proves Theorem \ref{inequality}. 
\medskip\noindent We are now left with proving \eqref{main claim}.

\medskip \noindent{\bf Step 3: Proof of \eqref{main claim}:} 
For any positive integer $m$, one can write that
\begin{align}
\Delta_{g}^{m} (\sqrt{\eta_{i}} u) = \sqrt{\eta_{i}} \Delta_{g}^{m}u + \mathcal{P}_{g}^{(2m-1,1)}(u, \sqrt{\eta_{i}}) + \mathcal{L}_{\sqrt{\eta_{i}},g}^{2m-2}(u)
\end{align}
where
\begin{align}
 \mathcal{P}_{g}^{(2m-1,1)}(u, \sqrt{\eta_{i}}) = \sum \limits_{|l|=2m-1,|\beta|=1} (a_{l,\beta}\partial_{\beta} \sqrt{\eta_{i}}) \nabla^{l} u,\hbox{ and }\displaystyle{  \mathcal{L}_{\sqrt{\eta_{i}},  g}^{2m-2}(u) = \sum \limits_{|l|=0}^{2m-2} a_{l}(\sqrt{\eta_{i}}) ~ \nabla^{l}u }
\end{align}
 the coefficients $a_{l,\beta}$ and $ a_{l}(\sqrt{\eta_{i}})$ are smooth functions on $M$. The $a_{l,\beta}$'s  depends  only on  the metric $g$ and  on the manifold $M$ and $ a_{l}(\sqrt{\eta_{i}})$'s  depends  both on the metric $g$,  the function $\sqrt{\eta_{i}}$ and its derivatives upto order $2m$.
We shall use the same notations  $\mathcal{P}_{g}^{(2m-1,1)}(u, \sqrt{\eta_{i}})$, $\mathcal{L}_{\sqrt{\eta_{i}},g}^{2m-2}(u)$ for any expression of the above form.

\medskip\noindent{\bf Step 3.1: $k$ is even.} We then write $k=2m$, $m\geq 1$, and then
\begin{eqnarray}
&& \sum \limits_{i=1}^{N}  \int \limits_{M}  \left( \Delta^{m}_{{g}} (\sqrt{\eta_{i}} u) \right)^2  ~{dv_{g}} = \sum \limits_{i=1}^{N}  \int \limits_{M}  {\eta_{i}} \left( \Delta^{m}_{{g}}  u \right)^2  ~{dv_{g}}  \notag \\
&& +\sum \limits_{i=1}^{N} \int_{M} \left(\mathcal{P}_{g}^{(2m-1,1)}(u, \sqrt{\eta_{i}}) \right)^{2} ~ dv_{g} \notag  + \sum \limits_{i=1}^{N} \int_{M} \left( \mathcal{L}_{\sqrt{\eta_{i}},  g}^{2m-2}(u) \right)^{2} ~ dv_{g} \notag \\
&& + 2 \sum \limits_{i=1}^{N} \int_{M} \sqrt{\eta_{i}} \Delta^{m}_{g}u ~\mathcal{P}_{g}^{(2m-1,1)}(u, \sqrt{\eta_{i}})~ dv_{g} + 2 \sum \limits_{i=1}^{N} \int_{M}\sqrt{\eta_{i}} \Delta^{m}_{g}u ~ \mathcal{L}_{\sqrt{\eta_{i}},  g}^{2m-2}(u)~dv_{g}\notag \\
&& + 2 \sum \limits_{i=1}^{N} \int_{M} \mathcal{P}_{g}^{(2m-1,1)}(u, \sqrt{\eta_{i}})~ \mathcal{L}_{\sqrt{\eta_{i}},  g}^{2m-2}(u)~ dv_{g} 
\end{eqnarray}
\noindent
We note that
\begin{align}
\sum \limits_{i=1}^{N} \int_{M} \left(\mathcal{P}_{g}^{(2m-1,1)}(u, \sqrt{\eta_{i}}) \right)^{2} ~ dv_{g} \leq  C\left\| u\right\|_{H_{2m-1}^{2}}^{2}.
~\text{ and }~
\sum \limits_{i=1}^{N} \int_{M} \left( \mathcal{L}_{\sqrt{\eta_{i}},  g}^{2m-2}(u) \right)^{2} ~ dv_{g}  \leq C\left\| u\right\|_{H_{2m-2}^{2}}^{2}.
\end{align}
\noindent
On the other hand 
\begin{align}
&\sum \limits_{i=1}^{N} \int_{M} \sqrt{\eta_{i}} \Delta^{m}_{g}u ~\mathcal{P}_{g}^{(2m-1,1)}(u, \sqrt{\eta_{i}})~ dv_{g}  =
\sum \limits_{i=1}^{N}   ~ \sum \limits_{|l|=2m-1} \sum \limits_{|\beta|=1} \int_{M} ( \sqrt{\eta_{i}} \Delta^{m}_{g}u)  (  (a_{l,\beta}\partial_{\beta} \sqrt{\eta_{i}}) \nabla^{l} u )  ~ dv_{g}   \notag \\
&= \frac{1}{2} \sum \limits_{i=1}^{N}   ~ \sum \limits_{|l|=2m-1} \sum \limits_{|\beta|=1}  \int_{M} ( \Delta^{m}_{g}u)    ((a_{l,\beta}\partial_{\beta} {\eta_{i}}) \nabla^{l} u )  ~ dv_{g} \notag \\
&  = \frac{1}{2}  \sum \limits_{|l|=2m-1} \sum \limits_{|\beta|=1}  \int_{M} ( \Delta^{m}_{g}u)  (  (a_{l,\beta} ~ \partial_{\beta} ( \sum \limits_{i=1}^{N} {\eta_{i}})) \nabla^{l} u )  ~ dv_{g}  =0    \notag \\
\end{align}

\noindent
while using the integration by parts formula  we obtain 
\begin{align}
 \sum \limits_{i=1}^{N} \int_{M}\sqrt{\eta_{i}} \Delta^{m}_{g}u ~ \mathcal{L}_{\sqrt{\eta_{i}},  g}^{2m-2}(u)~dv_{g} \leq C \left\| u\right\|_{H_{2m-1}^{2}}^{2}
\end{align}

\noindent
and by H\"older inequality
\begin{align}
\sum \limits_{i=1}^{N} \int_{M} \mathcal{P}_{g}^{(2m-1,1)}(u, \sqrt{\eta_{i}})~ \mathcal{L}_{\sqrt{\eta_{i}},  g}^{2m-2}(u)~ dv_{g} \leq C\left\| u\right\|_{H_{2m-1}^{2}}^{2}
\end{align} 
Hence if $k$ is even, then
\begin{align}
 \sum \limits_{i=1}^{N}  \int \limits_{M}  \left( \Delta^{m}_{{g}} (\sqrt{\eta_{i}} u) \right)^2  ~{dv_{g}} \leq  \int \limits_{M}  \left( \Delta^{m}_{{g}}  u \right)^2  ~{dv_{g}} + C\left\| u\right\|_{H_{2m-1}^{2}}^{2}.
\end{align} 
So we have the claim for $k$ even.

\medskip\noindent{\bf Step 3.2: $k$ is odd.} We then write $k= 2m+1$ with $m\geq 0$. We have
\begin{align}
\nabla \left( \Delta_{g}^{m} (\sqrt{\eta_{i}} u) \right) = \sqrt{\eta_{i}} ~ \nabla \left( \Delta_{g}^{m}u  \right)+ (\Delta_{g}^{m} u) ~\nabla \sqrt{\eta_{i}} + \nabla \left( \mathcal{P}_{g}^{(2m-1,1)}(u, \sqrt{\eta_{i}}) \right) +  \nabla \left(  \mathcal{L}_{\sqrt{\eta_{i}},g}^{2m-2}(u) \right)
\end{align}
and so 
\begin{align}
&\sum \limits_{i=1}^{N}  \int \limits_{M} \left| \nabla \left( \Delta_{g}^{m} (\sqrt{\eta_{i}} u) \right)  \right|^{2} dv_{g} =  \sum \limits_{i=1}^{N}  \int \limits_{M} \eta_{i} \left| \nabla \left( \Delta_{g}^{m}u  \right)  \right|^{2} ~dv_{g}  +\sum \limits_{i=1}^{N}  \int \limits_{M} (\Delta_{g}^{m} u)^{2} \left| \nabla \sqrt{\eta_{i}}  \right|^{2} ~dv_{g} \\
 &  +\sum \limits_{i=1}^{N}  \int \limits_{M} \left|  \nabla \left( \mathcal{P}_{g}^{(2m-1,1)}(u, \sqrt{\eta_{i}}) \right) \right|^{2}~ dv_{g}  +\sum \limits_{i=1}^{N}  \int \limits_{M} \left| \nabla \left(  \mathcal{L}_{\sqrt{\eta_{i}},g}^{2m-2}(u) \right) \right|^{2}~ dv_{g} \notag \\ 
&  +2\sum \limits_{i=1}^{N}  \int \limits_{M} (\sqrt{\eta_{i}} ~ \nabla \left( \Delta_{g}^{m}u  \right), (\Delta_{g}^{m} u) ~\nabla\sqrt{\eta_{i}}~)~ dv_{g}  +2\sum \limits_{i=1}^{N}  \int \limits_{M} (\sqrt{\eta_{i}} ~ \nabla \left( \Delta_{g}^{m}u  \right),\nabla \left( \mathcal{P}_{g}^{(2m-1,1)}(u, \sqrt{\eta_{i}}) \right))~ dv_{g}  \notag \\ 
 & +2\sum \limits_{i=1}^{N}  \int \limits_{M} (\sqrt{\eta_{i}} ~ \nabla \left( \Delta_{g}^{m}u  \right),  \nabla \left(  \mathcal{L}_{\sqrt{\eta_{i}},g}^{2m-2}(u) \right))~ dv_{g}  
  +2\sum \limits_{i=1}^{N}  \int \limits_{M} ((\Delta_{g}^{m} u) ~\nabla \sqrt{\eta_{i}},\nabla \left( \mathcal{P}_{g}^{(2m-1,1)}(u, \sqrt{\eta_{i}}) \right))~ dv_{g}  \notag \\ 
 & +2\sum \limits_{i=1}^{N}  \int \limits_{M} ((\Delta_{g}^{m} u) ~\nabla \sqrt{\eta_{i}}, \nabla \left(  \mathcal{L}_{\sqrt{\eta_{i}},g}^{2m-2}(u) \right))~ dv_{g}  +2\sum \limits_{i=1}^{N}  \int \limits_{M} (\nabla \left( \mathcal{P}_{g}^{(2m-1,1)}(u, \sqrt{\eta_{i}}) \right), \nabla  \left(  \mathcal{L}_{\sqrt{\eta_{i}},g}^{2m-2}(u) \right)) ~dv_{g}  
\end{align}
\medskip

\noindent
We have that
\begin{align}
\sum \limits_{i=1}^{N}  \int \limits_{M} \left|  \nabla \left( \mathcal{P}_{g}^{(2m-1,1)}(u, \sqrt{\eta_{i}}) \right) \right|^{2}~ dv_{g} 
\leq  C \left\| u\right\|_{H_{2m}^{2}}^{2} 
~ \text{ and }
\sum \limits_{i=1}^{N}  \int \limits_{M} \left| \nabla \left(  \mathcal{L}_{\sqrt{\eta_{i}},g}^{2m-2}(u) \right) \right|^{2}~ dv_{g}  
\leq  C \left\| u\right\|_{H_{2m-1}^{2}}^{2} 
\end{align}

while
\begin{align}
& \sum \limits_{i=1}^{N}  \int \limits_{M} (\sqrt{\eta_{i}} ~ \nabla \left( \Delta_{g}^{m}u  \right), (\Delta_{g}^{m} u) ~\nabla\sqrt{\eta_{i}})~ dv_{g}  
=   \sum \limits_{i=1}^{N}  \int \limits_{M} (  \nabla \left( \Delta_{g}^{m}u  \right), (\Delta_{g}^{m} u) ~(  \sqrt{\eta_{i}}~ \nabla \sqrt{\eta_{i}}) )~ dv_{g} \notag \\
 &= \frac{1}{2}  \sum \limits_{i=1}^{N}  \int \limits_{M} (  \nabla \left( \Delta_{g}^{m}u  \right), (\Delta_{g}^{m} u) ~ \nabla {\eta_{i}} )~ dv_{g}=  \frac{1}{2}   \int \limits_{M} (  \nabla \left( \Delta_{g}^{m}u  \right), (\Delta_{g}^{m} u) ~ \nabla (  \sum \limits_{i=1}^{N}  {\eta_{i}}) )~ dv_{g} =0 
\end{align}

\noindent
And we obtain 
\begin{align}
& \left|  \sum \limits_{i=1}^{N}  \int \limits_{M} (\sqrt{\eta_{i}} ~ \nabla \left( \Delta_{g}^{m}u  \right),\nabla \left( \mathcal{P}_{g}^{(2m-1,1)}(u, \sqrt{\eta_{i}}) \right))~ dv_{g}  \right| \notag \\
=& \left|  \sum \limits_{i=1}^{N} ~   \sum \limits_{|l|=2m-1} \sum \limits_{|\beta|=1}    \int \limits_{M} (\sqrt{\eta_{i}} ~ \nabla\left( \Delta_{g}^{m}u  \right),\nabla \left(  (a_{l,\beta}\partial_{\beta} \sqrt{\eta_{i}}) \nabla^{l} u   \right))~ dv_{g}  \right| \notag \\
 \leq &  \left| \sum \limits_{i=1}^{N} ~   \sum \limits_{|l|=2m} \sum \limits_{|\beta|=1}    \int \limits_{M} (\sqrt{\eta_{i}} ~ \nabla \left( \Delta_{g}^{m}u  \right),  (a_{l,\beta}\partial_{\beta} \sqrt{\eta_{i}}) \nabla^{l} u  )~ dv_{g} \right|  \notag\\ 
 & +  \left|  \sum \limits_{i=1}^{N}  ~  \sum \limits_{|l|=2m-1} \sum \limits_{|\beta|=1}    \int \limits_{M} (  \nabla \left( \Delta_{g}^{m}u  \right), \left( \sqrt{\eta_{i}} ~\nabla    (a_{l,\beta}\partial_{\beta} \sqrt{\eta_{i}}) \right) \nabla^{l} u    ) ~dv_{g}  \right|
  \notag\\
   \leq &  \left| \sum \limits_{i=1}^{N} ~   \sum \limits_{|l|=2m} \sum \limits_{|\beta|=1}    \int \limits_{M} (\sqrt{\eta_{i}} ~ \nabla\left( \Delta_{g}^{m}u  \right),  (a_{l,\beta}\partial_{\beta} \sqrt{\eta_{i}}) \nabla^{l} u  )~ dv_{g} \right|  \notag\\ 
 & +   \sum \limits_{i=1}^{N}  \left|   \sum \limits_{|l|=2m-1} \sum \limits_{|\beta|=1}    \int \limits_{B_{p_{i}} (\tau)} (  \nabla \left( \Delta_{g}^{m}u  \right), \left( \sqrt{\eta_{i}} ~\nabla    (a_{l,\beta}\partial_{\beta} \sqrt{\eta_{i}}) \right) \nabla^{l} u    ) ~dv_{g}  \right|
 \end{align}
Then we apply  the  integration by parts  formula on each of the domains $\varphi^{-1}(B_{p_{1}}(\tau)) \subset \R^{n} $ to obtain 
 \begin{align}
  &  \left| \sum \limits_{i=1}^{N} ~   \sum \limits_{|l|=2m} \sum \limits_{|\beta|=1}    \int \limits_{M} (\sqrt{\eta_{i}} ~ \nabla \left( \Delta_{g}^{m}u  \right),  (a_{l,\beta}\partial_{\beta} \sqrt{\eta_{i}}) \nabla^{l} u  )~ dv_{g} \right|  \notag\\ 
 & +   \sum \limits_{i=1}^{N}  \left|   \sum \limits_{|l|=2m-1} \sum \limits_{|\beta|=1}    \int \limits_{B_{p_{i}} (\tau)} (  \nabla\left( \Delta_{g}^{m}u  \right), \left( \sqrt{\eta_{i}} ~\nabla    (a_{l,\beta}\partial_{\beta} \sqrt{\eta_{i}}) \right) \nabla^{l} u    )~dv_{g}  \right| \notag\\
  \leq &  \left| \sum \limits_{i=1}^{N} ~   \sum \limits_{|l|=2m} \sum \limits_{|\beta|=1}    \int \limits_{M} (\sqrt{\eta_{i}} ~ \nabla \left( \Delta_{g}^{m}u  \right),  (a_{l,\beta}\partial_{\beta} \sqrt{\eta_{i}}) \nabla^{l} u  )~ dv_{g} \right|  +    C \left\| u\right\|_{H_{2m}^{2}}^{2}   \notag\\
  \leq &  \frac{1}{2} \left|  \sum \limits_{i=1}^{N} ~   \sum \limits_{|l|=2m} \sum \limits_{|\beta|=1}    \int \limits_{M} ( \nabla \left( \Delta_{g}^{m}u  \right),  (a_{l,\beta}\partial_{\beta} {\eta_{i}}) \nabla^{l} u  )~ dv_{g} \right|  +    C \left\| u\right\|_{H_{2m}^{2}}^{2}   \notag \\
    \leq &  \frac{1}{2} \left|    \sum \limits_{|l|=2m} \sum \limits_{|\beta|=1}    \int \limits_{M} ( \nabla \left( \Delta_{g}^{m}u  \right),  (a_{l,\beta} \partial_{\beta} (  \sum \limits_{i=1}^{N}  {\eta_{i}})) \nabla^{l} u  )~ dv_{g} \right|  +    C \left\| u\right\|_{H_{2m}^{2}}^{2}   \notag \\
 \leq &  C \left\| u\right\|_{H_{2m}^{2}}^{2} \qquad \text{since  \ \ $  \sum \limits_{i=1}^{N} \eta_{i} =1$ }
\end{align}
\medskip

\noindent
Similarly  after  integration by parts  one obtains
\begin{align}
\left| \sum \limits_{i=1}^{N}  \int \limits_{M} (\sqrt{\eta_{i}} ~ \nabla \left( \Delta_{g}^{m}u  \right),  \nabla \left(  \mathcal{L}_{\sqrt{\eta_{i}},g}^{2m-2}(u) \right))~ dv_{g}\right|  \leq  C \left\| u\right\|_{H_{2m}^{2}}^{2}  
\end{align}

\begin{align}
\left| \sum \limits_{i=1}^{N}  \int \limits_{M} ((\Delta_{g}^{m} u) ~\nabla \sqrt{\eta_{i}},\nabla \left( \mathcal{P}_{g}^{(2m-1,1)}(u, \sqrt{\eta_{i}}) \right))~ dv_{g} \right|  \leq  C \left\| u\right\|_{H_{2m}^{2}}^{2} 
\end{align}
\noindent
and 
\begin{align}
 & \sum \limits_{i=1}^{N}  \int \limits_{M} ((\Delta_{g}^{m} u) ~\nabla \sqrt{\eta_{i}}, \nabla \left(  \mathcal{L}_{\sqrt{\eta_{i}},g}^{2m-2}(u) \right))~ dv_{g}  \notag \\
&+ \sum \limits_{i=1}^{N}  \int \limits_{M} (\nabla \left( \mathcal{P}_{g}^{(2m-1,1)}(u, \sqrt{\eta_{i}}) \right), \nabla \left(  \mathcal{L}_{\sqrt{\eta_{i}},g}^{2m-2}(u) \right)) ~dv_{g}   \leq  C \left\| u\right\|_{H_{2m}^{2}}^{2} 
\end{align}
\medskip

\noindent
Hence for  $k$  odd, we also obtain that 
\begin{align}
 \sum \limits_{i=1}^{N}  \int \limits_{M} \left( \nabla \left( \Delta_{g}^{m} (\sqrt{\eta_{i}} u) \right)  \right)^{2} dv_{g} \leq   \int \limits_{M}  \left( \nabla \left( \Delta_{g}^{m}u  \right)  \right)^{2}_{g}  ~{dv_{g}} + C  \left\| u\right\|_{H_{2m}^{2}}^{2}.
\end{align}
Hence we have the claim and this completes the proof.



\section{Best constant and direct Minimizaton}\label{sec:min}
Let $\Omega_M\subset M$ be any smooth $n-$dimensional submanifold of $M$, possibly with boundary. In the sequel, we will either take $\Omega_M=M$, or $M\setminus \overline{B_{x_0}(\epsilon_0)}$ for some  $\epsilon_0>0$ small enough. We define $H_{k,0}^2(\Omega_M)\subset H_{k}^2(M)$ as the completion of $C^\infty_c(\Omega_M)$ for the norm $\Vert\cdot\Vert_{H_k^2}$. In this section, we prove the following result in the spirit of Aubin \cite{aubin}:

\begin{theorem}\label{th:min} Let $(M,g)$ be a compact Riemannian manifold of dimension $n>2k$, with $k\geq 1$. $\Omega_M\subset M$ be any smooth $n-$dimensional submanifold of $M$ as above. Let $P$ be a differential operator as in \eqref{the op} and let $f\in C^{0,\theta}(\Omega_M)$ be a H\"older continuous positive function. Assume that $P$ is coercive on $H_{k,0}^2(\Omega_M)$. Suppose  that
 \begin{equation}\label{ineq:large}
\inf \limits_{u \in  \mathcal{N}_f}   \int_{\Omega_M} u P(u)\, dv_g <  \frac{1}{\left( \sup_{\Omega_M} f \right)^{\frac{2}{\crit}}K_{0}(n,k)},
\end{equation}
where 
\begin{align}
\mathcal{N}_{f}:= \{ u \in H_{k,0}^{2}(\Omega_M) : \int \limits_{\Omega_M} f \left| u\right|^{2^{\sharp}_{k}}dv_{g}=1 \}.
\end{align}
Then there exists a minimizer $u\in \mathcal{N}_f$. Moreover, up to multiplication by a constant, $u\in C^{2k}(\overline{\Omega_M})$ is a solution to \begin{eqnarray*}
 \left \{ \begin{array} {lc}
          Pu  = f \left| u \right|^{ \crit - 2} u \qquad  \text{in} ~ \Omega_{M}\\
          D^{\alpha} u=0  \ \  \  \qquad \qquad\text{on } ~\partial \Omega_{M} \quad \text{for} \ \ |\alpha| \leq k-1.
            \end{array} \right. 
\end{eqnarray*} In addition, if the Green's function of $P$ on $\Omega_M$ with Dirichlet boundary condition is positive, then any minimizer is either positive or negative. When $\Omega_M=M$, and the Green's function of $P$ on $M$ is positive, then up to changing sign, $u>0$ is a solution to
$$P u= fu^{\crit-1}\hbox{ in }M.$$
\end{theorem}

\medskip\noindent{\it Proof of Theorem \ref{th:min}:} This type of result is classical. We only sketch the proof. For simplicity, we take $\Omega_M=M$. The proof of the general case is similar. Here and in the sequel, we define (see \eqref{def:upv})
$$I_P(u):=\int_M u P(u)\, dv_g\hbox{ for all }u\in H_k^2(M).$$
We start with the following lemma: 
\begin{lemma}\label{lem:min} Let $(u_{i}) \in \mathcal{N}_{f} $ be a minimizing sequence for $I_P$ on $\mathcal{N}_f$. Then
\begin{itemize}
\item[(i)] Either there exists $u_0\in \mathcal{N}_f$ such that $u_i\to u_0$ strongly in $H_{k}^{2}(M)$, and $u_0$ is a minimizer of $I_P$ on $\mathcal{N}_f$
\item[(ii)] Or there exists $x_0\in \overline{\Omega_M}$ such that $f(x_0)=\max_{\overline{\Omega_M}} f$ and $|u_i|^{\crit}\, dv_g\rightharpoonup \delta_{x_0}$ as $i\to +\infty$ in the sense of measures. Moreover, $\inf \limits_{u\in \mathcal{N}_f } I_P(u)=\frac{1}{K_0(n,k)(\max_M f)^{\frac{2}{\crit}}}$.
\end{itemize}
\end{lemma}


\smallskip\noindent{\it Proof of Lemma \ref{lem:min}:} We define $\alpha:=\inf\{I_P(u)/\, u\in {\mathcal N}_f\}$. As the functional $I_{g}$ is coercive so the sequence $(u_{i})$ is bounded in  $H^{2}_{k}(M)$. We let $u_0\in H^{2}_{k}(M)$ such that, up to a subsequence, $u_{i} \rightharpoonup u_{0}$ weakly in $H^{2}_{k}(M)$ as $i\to +\infty$, and $u_i(x)\to u_0(x)$ as $i\to +\infty$ for a.e. $x\in M$. Therefore,
\begin{align}\label{bnd:u0}
 \left\|u_0\right\|_{L^{2_{k}^{\sharp}}}^{2^{\sharp}_{k}} \leq  \liminf_{i\to +\infty}\left\|u_i\right\|_{L^{2_{k}^{\sharp}}}^{2^{\sharp}_{k}}=1. 
\end{align}
We define $v_{i}:=u_{i}-u_{0}$. Up to extracting a subsequence, we have that $(v_i)_i\to 0$ in $H_{k-1}^2(M)$. We define $\mu_{i} := (\Delta_{g}^{k/2} u_{i})^{2}~dv_{g}$ and $\tilde{\nu}_{i} = | u_{i}|^{2_{k}^{\sharp}}~dv_{g}$ and $\nu_{i} = f| u_{i}|^{2_{k}^{\sharp}}~dv_{g}$  for all $i$. Up to a subsequence, we denote respectively by $\mu$, $\tilde{\nu}$ and $\nu$ their limits in the sense of measures. It follows from the concentration-compactness Theorem \ref{th:CC} that, 
\begin{equation}
\tilde{\nu} = |u_{0}|^{2_{k}^{\sharp}}\, dv_g + \sum \limits_{j \in \mathcal{J}} \alpha_{j} \delta_{x_{j}} \hbox{ and }\mu \geq  (\Delta_{g}^{k/2} u_{0})^{2}\, dv_g + \sum \limits_{j \in \mathcal{I}} \beta_{j} \delta_{x_{i}}\label{ineq:mu}
\end{equation}
where $J\subset \mathbb{N}$ is at most countable,  $(x_{j})_{j\in J} \in M$ is a family of points, and $(\alpha_{j})_{j\in J}\in \mathbb{R}_{\geq 0}$, $(\beta_j)_{j\in J}\in \mathbb{R}_{\geq 0}$ are such that
\begin{align}
 \alpha_{j}^{2/{2_{k}^{\sharp}}} \leq K_{0}(n,k) ~ \beta_{j} \hbox{ for all }j\in J.
\end{align}
As a consequence, we get that
\begin{equation}
\nu = f|u_{0}|^{2_{k}^{\sharp}}\, dv_g + \sum \limits_{j \in \mathcal{J}} f(x_j)\alpha_{j} \delta_{x_{j}} \\
\end{equation}
Since $(u_{i}) \in \mathcal{N}_{f}$, and $M$ is compact, we have that $\int_M\, d\nu=1$ and then 
\begin{equation}\label{ineq:a}
1= \int_M f|u_0|^{\crit}\, dv_g+\sum \limits_{j \in \mathcal{J}} f(x_j)\alpha_{j}.
\end{equation}
Since $(u_i)_i\to u_0$ strongly in $H_{k-1}^2(M)$, integrating \eqref{ineq:mu} yields
\begin{equation}\label{ineq:b}
\alpha\geq I_P(u_0)+ \sum \limits_{j \in \mathcal{J}}  \beta_j\geq \alpha\Vert u_0\Vert_{\crit}^2+K_0(n,k)^{-1}\sum\limits_{j\in \mathcal {J}}\alpha_{j}^{2/{2_{k}^{\sharp}}}. 
\end{equation}
Since $\alpha\leq K_0(n,k)^{-1}(\max_M f)^{-2/\crit}$, we then get that 
\begin{itemize}
\item[(i)] either $\Vert u_0\Vert_{\crit}=1$ and $\alpha_j=0$ for all $j\in \mathcal{J}$,
\item[(ii)] or $u_0\equiv 0$, $f(x_{j_0})\alpha_{j_0}=1$ for some $j_0\in \mathcal{J}$, $f(x_{j_0})=\max_M f$ and $\alpha_{j}=0$ for all $j\neq j_0$.
\end{itemize}
In case (i), we get from the strong convergence to $0$ of $(v_i)_i$ in $H_{k-1}^2(M)$ that $I_P(u_i) =   \int \limits_{M}  (\Delta_g^{k/2} v_{i}  )^2 \, dv_g+ I_P(u_0)+o(1)$ as $i\to +\infty$. Since $u_0\in {\mathcal N}_f$ and $(u_i)$ is a minimizing sequence, we then get that $(v_i)_0$ goes to $0$ strongly in $H_k^2(M)$, and therefore $u_i\to u_0$ strongly in $H_k^2(M)$.

\smallskip\noindent In case (ii), \eqref{ineq:b} yields $\alpha=K_0(n,k)^{-1}(\max_M f)^{-2/\crit}$ and $I_P(u_0)=0$, which yields $u_0\equiv 0$ since the operator is coercive. 

\smallskip\noindent This completes the proof of Lemma \ref{lem:min}.\hfill$\Box$

\medskip\noindent We go back to the proof of Theorem \ref{th:min}. Let $(u_i)_i$ be a minimizing sequence for $I_P$ on ${\mathcal N}_f$. It follows from the assumption \eqref{ineq:large} that case (i) of Lemma \ref{lem:min} holds, and then, there exists a minimizer $u_0\in {\mathcal N}_f$ that is a minimizer. Therefore, it is a weak solution to $P_{g}^{k}u_0 = \alpha f \left| u_0 \right|^{\crit-2}u_0$ in $M$ (see \eqref{weak:sol} for the definition). It then follows from the regularity Theorem \ref{main regularity} that $u\in C^{2k,\theta}(M)$.

\medskip\noindent We let $G: M\times M\setminus\{(x,x)/\, x\in M\}$ be the Green's function of $P$ on $M$. We assume that $G(x,y)>0$ for all $x\neq y\in M$. Green's representation formula yields 
\begin{equation}\label{rep:green}
\varphi(x)=\int_M G(x,y) (P\varphi)(y)\, dv_g\hbox{ for all }x\in M\hbox{ and all }\varphi\in C^{2k}(M).
\end{equation}
 It follows from Proposition \ref{existence and uniqueness} that there exists $v \in H^{2}_{k}(M)$ such that 
\begin{align}
Pv = \alpha f \left| u_{0}\right|^{2^{\sharp}_{k}-1} \qquad \text{ in } M.
\end{align}
Standard regularity  (taking inspiration from Vand der Vorst \cite{vdv}) yields $v \in C^{2k}(M)$. We have that $P\left( v \pm u_{0} \right) \geq 0$. Since $G>0$, it follows from Green's formula \eqref{rep:green} that $  v \pm u_{0}  \geq 0$. So $v \geq |u_{0}|$ and  therefore $v \neq 0$. Independently, since $Pv\geq 0$ and $v\not\equiv 0$, Green's formula \eqref{rep:green} yields $v>0$. Using H\"older's inequality and $v\geq |u_0|$, we get that
\begin{eqnarray}
&& J_P(u)=\frac{ \int \limits_{{M}} v P(v)~ dv_{g}  }{\left( ~ \int \limits_{{M}} f |v|^{  2_k^{\sharp}}~ dv_{g} \right)^{2/  2_k^{\sharp}}} 
=\frac{ \alpha \int \limits_{{M}} v  f \left| u_{0}\right|^{2^{\sharp}_{k}-1} ~ dv_{g}  }{\left( ~ \int \limits_{{M}} f |v|^{  2_k^{\sharp}}~ dv_{g} \right)^{2/  2_k^{\sharp}}} \\
&&  \leq  \frac{ \alpha  \left( \int \limits_{{M}} f |v|^{  2_k^{\sharp}}~ dv_{g} \right)^{\frac{1}{2^{\sharp}_{k}}}  \left( \int \limits_{{M}} f |u_{0}|^{  2_k^{\sharp}}~ dv_{g} \right)^{\frac{2^{\sharp}_{k}-1}{2^{\sharp}_{k}}} }{\left( ~ \int \limits_{{M}} f |v|^{  2_k^{\sharp}}~ dv_{g} \right)^{2/  2_k^{\sharp}} }\\
&&  \leq  \frac{ \alpha   \left( \int \limits_{{M}} f |u_{0}|^{  2_k^{\sharp}}~ dv_{g} \right)^{\frac{2^{\sharp}_{k}-1}{2^{\sharp}_{k}}} }{\left( ~ \int \limits_{{M}} f |u_{0}|^{  2_k^{\sharp}}~ dv_{g} \right)^{1/  2_k^{\sharp}} }    \leq   \alpha   \left( \int \limits_{{M}} f |u_{0}|^{  2_k^{\sharp}}~ dv_{g} \right)^{\frac{2^{\sharp}_{k}-2}{2^{\sharp}_{k}}} \leq \alpha
\end{eqnarray} 
since  $\int \limits_{M} f \left| u_{0} \right|^{2^{\sharp}_{k}}dv_{g}=1$. Since $\alpha$ is the infimum of the functional, we get that $J_P(u)=\alpha$. Hence $v$  attains the infimum and therefore it also solves the equation $Pv = \mu f v^{2^{\sharp}_{k}-1}$ weakly in $M$, and $v\in H_{k,0}^2(M)$. Moreover, one has equality in all the inequalities above, and then $|u_{0}|=v>0$, and therefore either $u_{0}>0$ or $u_{0}<0$ in $M$. This ends the proof of Theorem \ref{th:min}.\hfill$\Box$




\section{Concentration Compactness Lemma }{\label{CC}}





We now state and prove the concentration  compactness lemma in the spirit of P.-L.Lions for  the case of a closed manifold:

\begin{theorem}[Concentration-compactness]\label{th:CC}
Let $(M,g)$  be a  smooth, compact Riemannian manifold of dimension $n$ and let $k$ be a positive integer such that $2k < {n}$. Suppose  $(u_{m})$ be  a bounded sequence in $H_{k}^{2}(M)$. Up to extracting a subsequence, there exist two nonnegative Borel-regular measure $\mu,\nu$ on $M$ and $u\in H_k^2(M)$ such that
\begin{enumerate}
\item
[(a)]   $u_m \rightharpoonup u \quad \text{weakly in $H_k^{2}(M)$}$\\
\item
[(b)] $\mu_{m} := (\Delta_{g}^{k/2} u_{m})^{2}~dv_{g} \rightharpoonup \mu \quad \text{weakly in the sense of measures} $\\
\item
[(c)] $\nu_{m} := | u_{m}|^{2_{k}^{\sharp}}~dv_{g} \rightharpoonup \nu \quad \text{weakly in the sense of measures} $\\
\end{enumerate}
Then there exists an at most countable index set $\mathcal{I}$, a family of distinct points $\{ x_{i} \in M : i \in \mathcal{I} \}$,  families  of nonnegative weights $\{ \alpha_{i} : i\in \mathcal{I}\}$ and $\{ \beta_{i} : i\in \mathcal{I}\}$ such that
\begin{enumerate}
\item[(i)]
\begin{align}
\nu =& |u|^{2_{k}^{\sharp}}\, dv_g + \sum \limits_{i \in \mathcal{I}} \alpha_{i} \delta_{x_{i}} \\
\mu \geq&  (\Delta_{g}^{k/2} u)^{2}\, dv_g + \sum \limits_{i \in \mathcal{I}} \beta_{i} \delta_{x_{i}}
\end{align}
where $\delta_{x} $ denotes the Dirac mass at $x \in M$ with mass equal to $1$.\item
[(ii)] for all $i \in \mathcal{I}$, $\alpha_{i}^{2/{2_{k}^{\sharp}}} \leq K_{0}(n,k) ~ \beta_{i} $. In particular $\sum \limits_{i \in \mathcal{I}} \alpha_{i}^{2/{2_{k}^{\sharp}}} <   \infty$.
\end{enumerate}
\end{theorem}

\medskip\noindent{\it Proof of Theorem \ref{th:CC}:} By the Riesz representation theorem $(\mu_{m})$, and $(\nu_{m})$  are sequences of   Radon  measures  on $M$. 

\medskip\noindent{\bf Step 1:} First we assume that $u \equiv 0$.
\medskip\noindent
Let $\varphi \in C^{\infty}(M)$, then from \eqref{optimal inequality} we have that, given any $ \varepsilon >0$ there exists $B_{\varepsilon} \in  \R$ such that
\begin{align}
\left( \int \limits_{M} |\varphi u_{m}|^{2_{k}^{\sharp}} ~ dv_{g} \right)^{2/{2_{k}^{\sharp}}} \leq \left(K_{0}(n,k) + \varepsilon \right) \int \limits_{M} (\Delta_{g}^{k/2} (\varphi u_{m}))^{2}~ dv_{g} + B_{\varepsilon} || \varphi u_{m}||^{2}_{H_{k-1}^{2}}.
\end{align} 
\medskip

\noindent
Since $u_{m } \rightharpoonup 0$ in $H_{k}^{2}(M)$, letting $m \rightarrow + \infty$ and then taking the limit $\varepsilon \rightarrow 0$, it follows that 
\begin{align}{\label{borel inequality}}
\left( \int \limits_{M} |\varphi |^{2_{k}^{\sharp}} ~ d\nu \right)^{2/{2_{k}^{\sharp}}} \leq K_{0}(n,k)  \int \limits_{M} \varphi^{2}~ d\mu. 
\end{align} 
\medskip

\noindent
By regularity of the Borel measure  $\nu $,   $(\ref{borel inequality})$ holds for any Borel measurable function $\varphi$ and in particular   for any Borel set $E \subset M$  we have 
\begin{align}{\label{borel}}
\nu(E)^{2/{2_{k}^{\sharp}}}  \leq K_{0}(n,k)  ~\mu(E).
\end{align}
Therefore the measure $\nu$ is absolutely continuous with respect to the measure $\mu$ and hence by the Radon-Nikodyn theorem, we get 
\begin{equation}
d\nu =  f d \mu \hbox{ and }d \mu = g d\nu + d \sigma
\end{equation}
where $f \in L^{1}(M, \mu)$ and $g \in L^{1}(M, \nu)$ are nonnegative functions, $\sigma$ is a positive Borel measure on $M$ and $d\nu \bot d \sigma$. 
\medskip

\noindent
Let $S= M \backslash(supp ~ \sigma)$. Then for any $\varphi \in C(M)$  with support $supp(\varphi) \subset S$ one has
\begin{align}{\label{154}}
\int \limits_{M} \varphi ~  d\nu =  \int \limits_{M} \varphi f ~ d \mu =  \int \limits_{M} \varphi ~   fg~ d\nu. 
\end{align}
By regularity of the Borel measures $\mu$ and $\nu $  $(\ref{154})$ holds for any Borel measurable function $\varphi$.
This implies that $fg=1$ a.e with respect to $\nu$. So, in particular $g >0$  $\nu$ a.e in $S$. Let $\psi  \in C(M)$, taking $\varphi=  \psi \chi_{S}$ in $(\ref{borel inequality})$ we have
\begin{eqnarray}
&&\left( \int \limits_{M} |\psi |^{2_{k}^{\sharp}}  \mathcal{X}_{S} ~ d\nu \right)^{2/{2_{k}^{\sharp}}} \leq  K_{0}(n,k)  \int \limits_{M} \psi^{2}  \mathcal{X}_{S}~ d\mu  \notag\\
&&  = K_{0}(n,k)  \int \limits_{M} \psi^{2}  \mathcal{X}_{S}~ [g d\nu + d \sigma]   =K_{0}(n,k)  \int \limits_{M} \psi^{2} g \mathcal{X}_{S}~ d\nu  
\end{eqnarray} 
Since $d\nu \bot d \sigma$ and $supp~ \nu \subset S$, we get that
\begin{align}{\label{g borel inequality}}
\left( \int \limits_{M} |\psi |^{2_{k}^{\sharp}} ~ d\nu \right)^{2/{2_{k}^{\sharp}}} \leq K_{0}(n,k)  \int \limits_{M} \psi^{2}g~ d\nu
 \end{align}
 By regularity of the Borel measure  $\nu $  the above relation  holds for any Borel measurable function $\psi$.
 \medskip

 \noindent
 Let $\phi  \in C(M)$ and let   $\displaystyle{\psi= \phi g^{\frac{1}{2_{k}^{\sharp}-2}} \mathcal{X}_{ \{g \leq N\} }}$ , $ \displaystyle{d\nu_{N}= g^{\frac{2_{k}^{\sharp}}{2_{k}^{\sharp}-2}} \mathcal{X}_{ \{g \leq N\} }} d\nu$. Then we have
 \begin{align}
 \left( \int \limits_{M} |\phi |^{2_{k}^{\sharp}} ~ d\nu_{N} \right)^{2/{2_{k}^{\sharp}}} \leq K_{0}(n,k)  \int \limits_{M} \phi^{2}~ d\nu_{N}.
 \end{align}
  By regularity of the Borel measure  $\nu $  the above relation  holds for any Borel measurable function $\phi$.
 
  \noindent
 It follows from Proposition \ref{borellemma} below that for each $N$ there exist a finite set $\mathcal{I}_{N}$,  a finite set of  distinct points $\{ x_{i} : i \in \mathcal{I}_{N}\}$and  a finite set of weights $\{ \tilde{\alpha}_{i} : i \in \mathcal{I}_{N}\}$ such that
 \begin{align}
 d{\nu}_{N}= \sum \limits_{i \in \mathcal{I}_{N}} {\tilde{\alpha}}_{i} ~ \delta_{x_{i}}
 \end{align}
  Let $\mathcal{I} = \bigcup \limits_{N=1}^{\infty} \mathcal{I}_{N}$. Then $\mathcal{I}$ is a countable set. For a Borel set  $E$, then one has  by monotone convergence theorem
 \begin{align}
  \int \limits_{M} \chi_{E} ~  g^{\frac{2_{k}^{\sharp}}{2_{k}^{\sharp}-2}}  d\nu =
 \lim \limits_{N \rightarrow \infty}  \int \limits_{M} \chi_{E} ~ d\nu_{N}.
 \end{align}
So $g^{\frac{2_{k}^{\sharp}}{2_{k}^{\sharp}-2}}  d\nu=  \sum \limits_{i \in \mathcal{I}} {\tilde{\alpha}}_{i} \delta_{x_{i}}$. 
 Since $g >0$ $\nu$ a.e , there exists $\alpha_{i}>0$  such that we have $ d\nu=  \sum \limits_{i \in \mathcal{I}} {{\alpha}}_{i} \delta_{x_{i}}$. Since $\mu = g d \nu + d \sigma \geq  g d \nu $, we get that 
 \begin{align}
 \mu \geq   \sum \limits_{i \in \mathcal{I}} {\beta}_{i} \delta_{x_{i}}   \qquad \text{  where } \beta_{i}=  g(x_{i}){{\alpha}}_{i} 
 \end{align}
 \medskip

 \noindent
 Taking $\psi=\mathcal{X}_{\{ x_{i}\}}$ in $( \ref{g borel inequality})$ we have for all $i \in \mathcal{I}$
 \begin{align}
  \alpha_{i}^{2/{2_{k}^{\sharp}}}  \leq K_{0}(n,k) ~ g  \alpha_{i}= K_{0}(n,k) ~\beta_{i}
 \end{align}
 and 
\begin{align}
\frac{1}{K_{0}(n,k) }\sum \limits_{i \in \mathcal{I}}    \alpha_{i}^{2/{2_{k}^{\sharp}}}  \leq   \sum \limits_{i \in \mathcal{I}}  \beta_{i} \leq   \mu(M) < + \infty.
\end{align}
 \medskip

 \noindent
 This proves the theorem for $u \equiv 0$. This ends Step 1.

 \medskip\noindent{\bf Step 2:} Assume $u \not\equiv 0$
and let $v_{m}:= u_{m}-u$. Then  $v_{m} \rightharpoonup 0$ weakly in $H_{k}^{2}(M)$. Therefore, as one checks, $\tilde{ \mu}_{m} := (\Delta_{g}^{k/2} v_{m})^{2}~dv_{g} \rightharpoonup  \mu -  (\Delta_{g}^{k/2} u)^{2}~dv_{g}$ and $\tilde{\nu}_{m} := | v_{m}|^{2_{k}^{\sharp}}~dv_{g} \rightharpoonup \nu - | u|^{2_{k}^{\sharp}}~dv_{g} $ weakly in the sense of measures. Applying Step 1 to the measures  $\tilde{\mu}_{m}$ and $\tilde{\nu}_{m}$ yields Theorem \ref{th:CC}.\hfill$\Box$
 
 
 
\medskip\noindent We now prove the reversed  H\"{o}lder inequality that was used in the proof.
\begin{proposition}\label{borellemma}
Let $\mu$ be a finite Borel measure on $M$ and suppose that  for any  Borel measurable function $\varphi$ one has 
\begin{align}
\left( \int \limits_{M} |\varphi|^{q}~ d\mu\right)^{1/q} \leq C \left( \int \limits_{M} |\varphi|^{p}~ d\mu\right)^{1/p}
\end{align}
for some $C>0$ and $1 \leq p < q< + \infty$. Then there exists j points $x_{1}, \ldots, x_{j} \in M$, and j positive real numbers $c_{1}, \ldots, c_{j}$ such that
\begin{align}
\mu = \sum \limits_{i=1}^{j} c_{i} \delta_{x_{i}}
\end{align}
where $\delta_{x} $ denotes the Dirac measure concentrated at $x \in M$ with mass equal to $1$. Moreover $c_{i} \geq (\frac{1}{C})^{\frac{pq}{q-p}}$. 
\end{proposition}
\begin{proof}
Let $E$ be a Borel set  in $M$. Taking $\varphi= \chi_{E}$ we obtain that, either $\mu(E)=0$ or $\mu(E) \geq (\frac{1}{C})^{\frac{pq}{q-p}} $
\medskip

\noindent
We define $\mathcal{O}:= \{ x\in M: \text{for some $r>0$ }~ \mu(B_{x}(r))=0 \}$. Then $\mathcal{O}$ is open. Now if $K \subset \mathcal{O}$ is compact, then $K$ can be covered by a finite number of balls each of which has measure $0$,  therefore $\mu(K)=0$. By the regularity of the measure hence it follows that $\mu(\mathcal{O})=0$. If $x \in M \backslash {\mathcal{O}} $, then for all $r >0$ one has $\mu(B_{x}(r)) \geq (\frac{1}{C})^{\frac{pq}{q-p}}$. Then
\begin{align}
\mu(\{x \}) = \lim \limits_{m\rightarrow + \infty} \mu(B_{x}({1}/{m}))   \geq \left( \frac{1}{C} \right)^{\frac{pq}{q-p}}.
\end{align}
Since the measure $\mu$ is finite, this implies that  that the set $M \backslash {\mathcal{O}}$ is finite. So  let $M \backslash {\mathcal{O}}= \{ x_{1}, \cdots, x_{j}\}$, therefore for any borel set $E$ in $M$
\begin{align}
\mu(E)= \mu(E \cap\{ x_{1}, \cdots, x_{j}\} )= \sum \limits_{x_{i} \in E} \mu(\{x_{i}\})=  \sum \limits_{i=1}^{j} \mu(\{x_{i} \}) \delta_{x_{i}}(E).
\end{align}
Hence the lemma follows with $c_{i}= \mu(\{x_{i}\})$.
\end{proof}

\section{Topological method of Coron}\label{sec:coron}
In this section  we obtain   higher energy solutions  by Coron's topological method if  the functional $J_{P}$ does not have   a minimizer,  for the case  $f  \equiv 1$. This will complete the  proof of  the first part of Theorem \ref{main}, that is the existence of solutions to \eqref{eq:coron} with no sign-restriction. For $\mu > 0$ and $y_{0} \in \R^{n}$, we define 


 \begin{equation}
   \mathcal{B}_{y_0,\mu}(y)= \alpha_{n,k}\left( \frac{\mu}{\mu^{2}+ \left|y-y_0 \right|^2}\right)^{\frac{n-2k}{2}}
 \end{equation}
where the choice of $\alpha_{n,k}$'s are such that for all $\mu$, $ \left\|  {\mathcal{B}}_{y_0,\mu} \right\|_{L^{2_{k}^{\sharp}}}=1$ and   
$\left\|  {\mathcal{B}}_{y_0,\mu} \right\|_{\mathscr{D}^{k,2}}^{2}=\frac{ 1 }{K_{0}(n,k)} $.  These functions are the extremal functions of the Euclidean Sobolev Inequality \eqref{sinq} and they  satisfy  the equation
\begin{align}
\Delta^{k} \mathcal{B}_{y_0,\mu} = \frac{1}{K_{0}(n,k)} \mathcal{B}_{y_0,\mu}^{{2_{k}^{\sharp}}-1} \qquad \text{in} ~\R^{n}.
\end{align}.
\medskip

\noindent
Let  $\tilde{\eta}_{r} \in C_{c}^{\infty}(\R^{n})$, $0 \leq  \tilde{\eta}_{r} \leq 1$ be a smooth cut-off function, such that $\tilde{\eta} _{r} =1$ for $x \in B_0(r)$ and $\tilde{\eta} _{r} =1$ for $x \in {\R}^n \backslash B_0(2 r)$. Let $\iota_{g}>0$ be the injectivity radius of $(M,g)$. For any $p \in {M}$, we let $\eta_{p}$ be a smooth cut-off function on $M$ such that
\begin{eqnarray}
\eta_{p}(x) = \left \{ \begin{array} {lc}
                 \tilde{ \eta} _{\frac{\iota_{g}}{10}} (exp_p^{-1}(x)) \quad \text{for } ~  x \in   B_p(\iota_g) \subset M\\
                 \qquad 0 \qquad \quad \quad   \text{ \ for } ~ x \in M \backslash B_p(\iota_g).
          \end{array} \right.
\end{eqnarray}
For any $x\in M$, we define 
\begin{align}
\mathcal{B}_{p, \mu}^{M}(x)=  \eta_{p}(x) ~\mathcal{B}_{0,\mu}(exp_p^{-1}(x)).
\end{align}
$\mathcal{B}_{p, \mu}^{M}$ is the standard bubble centered at the point $p \in M$ and with radius $\mu$ 
\begin{align}
\mathcal{B}_{p,\mu}^{M}(x)=  \alpha_{n,k} \eta_{p}(x) \left( \frac{\mu}{\mu^{2}+ {d_{g}(p,x)}^2}\right)^{\frac{n-2k}{2}}.
\end{align}
\medskip

\noindent
We  have
 \begin{proposition}{\label{bubble energy}}
 Let $(M,g)$ be a  smooth, compact Riemannian manifold  of dimension $n$ and let $k$ be a positve integer such that $2k < {n}$.  
 Consider the functional $J_P$ on the space $H_{k}^{2} (M)\backslash \{0\}$. Then  the sequence of functions $\left( \mathcal{B}_{p,\mu}^{M} \right)\in C^{\infty}(M)$  defined above is such that:\\
 $$\begin{array}{lll}
 (a) & \lim \limits_{\mu  \rightarrow 0} J_P ( \mathcal{B}_{p,\mu}^{M} )= \frac{ 1 }{K_{0}(n,k)} & \text{uniformly for $p \in M$ }\\
  
(b) & \lim \limits_{\mu \rightarrow 0}\left\|\mathcal{B}_{p,\mu}^{M}\right\|_{{L^{2_{k}^{\sharp}}}} = 1 & \text{uniformly for $p \in M$}\\
 
(c) & \mathcal{B}_{p,\mu}^{M} \rightharpoonup 0~ &\text{weakly in $H_{k}^{2}(M)$, as $\mu \rightarrow 0$}
 \end{array}$$
 \end{proposition}
 \smallskip\noindent{\it Proof of Proposition \ref{bubble energy}:} We claim that (c) holds. We first prove that $\mathcal{B}_{p,\mu}^{M}$ is uniformly bounded in $H_k^2(M)$. Indeed,
\begin{eqnarray*}
\sum_{\alpha\leq k}\int \limits_{ M} \left(\Delta_g^{\alpha/2}~\mathcal{B}_{p,\mu}^{M}\right)^2~{dv_g}&\leq & \sum_{\alpha\leq k}\int \limits_{ B_{p}(\iota_g/5)}\left(\Delta_g^{\alpha/2}~\mathcal{B}_{p,\mu}^{M}\right)^2~{dv_g}\\
&\leq & C\sum_{l\leq k}\int \limits_{ B_{0}(\iota_g/5)} \left|\nabla^l~\mathcal{B}_{p,\mu}^{M}\circ \hbox{exp}_p\right|^2\, dx\\
&\leq & \sum_{l\leq k}\int \limits_{ B_{0}(\iota_g/5 )} \left|\nabla^l\left(\frac{\mu}{\mu^2+|x|^2}\right)^{\frac{n-2k}{2}}\right|^2\, dx\\
&\leq & \sum_{l\leq k}\int \limits_{ B_0(\iota_g/(5\mu))} \mu^{2(k-l)}\left|\nabla^l\left(1+|x|^2\right)^{-\frac{n-2k}{2}}\right|^2\, dx.
\end{eqnarray*} 
As one checks, the right-hand-side is uniformly bounded wrt $\mu\to 0$, so $(\mathcal{B}_{p,\mu}^{M})$ is uniformly bounded wrt $p$ and $\mu\to 0$. Moreover, the above computations yield $\int_M (\mathcal{B}_{p,\mu}^{M})^2\, dv_g\to 0$ as $\mu\to 0$. Therefore, $\mathcal{B}_{p,\mu}^{M}\rightharpoonup 0$ as $\mu\to 0$ uniformly wrt $p\in M$. This proves the claim.

\medskip\noindent The space $H_{k}^{2}(M)$  is compactly  embedded in $H_{k-1}^{2}(M)$. Therefore $\mathcal{B}_{p,\mu}^{M}  \to 0$ in $H_{k-1}^{2}(M)$ as $\mu \rightarrow 0 $. Hence 
\begin{align}
\lim \limits_{\mu \rightarrow 0} \left(  \sum \limits_{l =0}^{k-1}  \int \limits_{M} A_{l}(g) (\nabla^{l}  \mathcal{B}_{p,\mu}^{M}, \nabla^{l}  \mathcal{B}_{p,\mu}^{M} )~dv_{g} \right) =0. 
\end{align}

 \noindent
 Now we estimate the term $\int \limits_{M} | ~\mathcal{B}_{p,\mu}^{M} |^{  \crit} ~dv_{g}$. We claim that 
 \begin{align}
\lim \limits_{R \rightarrow + \infty}  \lim \limits_{\mu \rightarrow 0} \int \limits_{{M}\backslash B_{p}({\mu R})} | ~\mathcal{B}_{p,\mu}^{M} |^{  \crit} ~dv_{g} =0.
 \end{align}
 We fix $R$. Now for $\mu$ sufficiently small
 \begin{align}
 \int \limits_{{M}\backslash B_{p}({\mu R})} | ~\mathcal{B}_{p,\mu}^{M} |^{  \crit} ~dv_{g} &= 
 \int \limits_{{B_{p}(\iota_{g})}\backslash B_{p}({\mu R})} | ~\mathcal{B}_{p,\mu}^{M} |^{  \crit} ~dv_{g} \notag \\
 & =  \int \limits_{{B_{0}( \iota_{g})}\backslash B_{0}({\mu R})} | ~\mathcal{B}_{p,\mu}^{M}(exp_{p} (y)) |^{  \crit}~ \sqrt{|g(exp_{p} (y)) |}~ dy  \notag\\
&  \leq \int \limits_{{B_{0}( \frac{\iota_{g}}{\mu})}\backslash B_{0}(R)} |\mathcal{B}_{0,1}(y)|^{  \crit}~ \sqrt{|g(exp_{p} ({\mu y})) |}~ dy.  
 \end{align}
 Since $\mathcal{B}_{0,1}\in L^{\crit}(\rn)$, this yields the claim.

 \medskip\noindent
 Similarly,   for $\mu$ sufficiently small
 \begin{align}
 \int \limits_{ B_{p}({\mu R}  )}  | ~\mathcal{B}_{p,\mu}^{M} |^{  \crit} ~dv_{g} & = \int \limits_{ B_{0}( \mu {R}  )} | ~\mathcal{B}_{p,\mu}^{M}(exp_{p} (y)) |^{  \crit}~ \sqrt{|g(exp_{p} (y)) |}~ dy \\
& = \int \limits_{ B_{0}(R)} | ~\mathcal{B}_{0,1}|^{  \crit}~ \sqrt{|g(exp_{p} (\mu y)) |}~ dy \\
& = \int \limits_{ B_{0}(R)} | ~\mathcal{B}_{0,1}|^{  \crit}~ dy +  o\left( \left\|\mathcal{B}_{0,1}\right\|_{{L^{2_{k}^{\sharp}}}} \right) \qquad \text{as $\mu \rightarrow 0$}
 \end{align}
 \medskip
 
 \noindent
Therefore 
\begin{align}
  \lim \limits_{\mu \rightarrow 0} \int \limits_{M} | ~\mathcal{B}_{p,\mu}^{M} |^{  \crit} ~dv_{g} =    \int \limits_{\R^{n}} | ~\mathcal{B}_{0,1} |^{  \crit} ~dv_{g}.
\end{align}
So we have $(b)$.
\medskip

 \noindent
Finally  we estimate the term $ \int \limits_{M} (\lap_g^{k/2}~\mathcal{B}_{p,\mu}^{M})^2~ dv_{g}$. We fix $R>0$. By calculating in terms of the  local coordinates  given by $exp_{p}$, we get  for $\mu$ sufficiently small
 \begin{align}
 \int \limits_{ B_{p}(\mu {R})}  (\lap_g^{k/2}~\mathcal{B}_{p,\mu}^{M})^2~ dv_{g} =  \int \limits_{B_{0}(R)} (\Delta^{k/2} \mathcal{B}_{0,1})^{2}~ dy + o\left(1\right) \qquad \text{as $\mu \rightarrow 0$}.\label{lim:1}
 \end{align}
We claim that
 \begin{align}
\lim \limits_{R \rightarrow + \infty}  \lim \limits_{\mu \rightarrow 0} \int \limits_{{M}\backslash B_{p}(\mu {R})}  (\lap_g^{k/2}~\mathcal{B}_{p,\mu}^{M})^2~ dv_{g} =0.\label{lim:2}
 \end{align}
We prove the claim. Indeed, via the exponential map at $p$, we have that 
\begin{eqnarray}
&&\int \limits_{{M}\backslash B_{p}(\mu {R})}  (\lap_g^{k/2}~\mathcal{B}_{p,\mu}^{M})^2~ dv_{g}  =  \int \limits_{{B_{p}(\iota_{g})}\backslash B_{p}(\mu {R})}  (\lap_g^{k/2}~\mathcal{B}_{p,\mu}^{M})^2~ dv_{g}\\
&& = \int \limits_{{B_{0}(\iota_{g})}\backslash B_{0}(\mu {R})}  (\lap_{exp_p^*g}^{k/2}~\mathcal{B}_{0,\mu})^2~ dv_{exp_p^*g}\\
&&\leq C \sum \limits_{|\alpha|=0}^{k} ~  \int \limits_{{B_{0}(\iota_{g})}\backslash B_{0}(\mu {R})}  |D^{\alpha}( \tilde{\eta}_{\frac{\iota_{g}}{10}} \mathcal{B}_{0,\mu})|^2~ dx  
\end{eqnarray}
Since $\mathcal{B}_{0,\mu}\to 0$ strongly in $H_{k-1,loc}^2(\rn)$, then, as $\mu\to 0$, we have that
\begin{eqnarray}
&&\int \limits_{{M}\backslash B_{p}(\mu {R})}  (\lap_g^{k/2}~\mathcal{B}_{p,\mu}^{M})^2~ dv_{g} \leq C\int \limits_{{B_{0}(\iota_{g})}\backslash B_{0}(\mu {R})}   \tilde{\eta}_{\frac{\iota_{g}}{10}} ^2|D^{k}\mathcal{B}_{0,\mu}|^2~ dx+o(1)\\
&& \leq C\int \limits_{{B_{0}(\iota_{g}/\mu)}\backslash B_{0}({R})}  |D^{k} \mathcal{B}_{0,1}|^2~ dx+o(1) \leq C\int \limits_{\rn\backslash B_{0}({R})}  |D^{k} \mathcal{B}_{0,1}|^2~ dx+o(1).
\end{eqnarray}
Since $D^k\mathcal{B}_{0,1}\in L^2(\rn)$, this yields \eqref{lim:2}. This proves the claim.

\medskip\noindent Equations \eqref{lim:1} and \eqref{lim:2} yield (a) and (b) of Proposition \ref{bubble energy} for any fixed $p\in M$. Since the manifold $M$ is compact,  we note that in the above  calculations there is no dependence on the point $p $ of the closed manifold $M$. So the convergence is uniform  for all points $p \in M$. This ends the proof of Proposition \ref{bubble energy}.\hfill$\Box$

 
 \medskip\noindent
 Fix  some $\theta$ such that   $ \frac{ 1 }{K_{0}(n,k)} + 4\theta < 2^{2k/n} \frac{ 1 }{K_{0}(n,k)} $.  Then from \eqref{bubble energy}  it follows that, there exists   $\mu_{0}$ small, such that   for all $\mu \in (0, \iota_{g} \mu_{0}) $ and  for all $p \in  M$ we have 
\begin{align}{\label{58}}
J_P ( \mathcal{B}_{p,\mu}^{M} ) \leq \frac{ 1 }{K_{0}(n,k)} + \theta.
\end{align}
\medskip\noindent We fix $x_0\in M$, and we assimilate isometrically $T_{x_0}M$ to $\R^n$, and we define the sphere $S^{n-1}:=\{x\in \R^n/\, \Vert x\Vert =1\}$. For $(\sigma, t) \in S^{n-1} \times [0,\iota_{g}/2)$, we define $\sigma_{t}^{M}: =  exp_{x_{0}}({t}\sigma)$ and
\begin{align}
u_{t}^{\sigma}(x)=    \alpha_{n,k} \eta_{\sigma_{t}^{M}}(x) \left[ \frac{   \mu_{0} (\iota_{g}/2-t)  }{\left( \mu_{0} (\iota_{g}/2-t) \right)^{2}+ {d_{g}(\sigma_{t}^{M} ,x)}^2}\right]^{\frac{n-2k}{2}}=  \mathcal{B}_{\sigma_{t}^{M},\mu_{0} (\iota_{g}/2-t)}^{M} 
\end{align}
\medskip\noindent
It then follows from our previous step and the choice of $\mu_{0}$ in $(\ref{58})$
\begin{align}
J_P ( u_{t}^{\sigma} ) \leq \frac{ 1 }{K_{0}(n,k)} + \theta ~ \qquad  \forall (\sigma, t) \in S^{n-1} \times [0,\iota_{g}/2).
\end{align}
\medskip

 \noindent
 Let $\eta \in C_{c}^{\infty}(\R^{n})$ be  a smooth, nonnegative, cut-off function such that $\eta(x)=1$ for $|x| \geq {1}/{2}$ and $\eta(x)=0$ for $ |x| <  {1}/{4}$. For $R  \geq 1$, let $\eta_{R} $ be  a smooth, nonnegative, cut-off function, such that 
\begin{eqnarray}
{\eta}_{R}(x) = \left \{ \begin{array} {lc}
              \qquad \qquad 1  \qquad  \qquad \text{if  } \quad \qquad   d_{g}(x_{0},x) \geq \frac{\iota_{g}}{10R}   \\
             \eta \left(\frac{10R}{\iota_{g}} exp_{x_{0}}^{-1} (x)  \right) \quad \text{if } \quad  \qquad d_{g}(x_{0},x)  <  \frac{\iota_{g}}{10R}
          \end{array} \right.
\end{eqnarray}
 \medskip
 
\noindent
Then the functions $\eta_{R}$ are such that $\eta_R(x)=1$ if $d_{g}(x_{0},x) \geq \frac{\iota_{g}}{20R}$ and $\eta_R(x)=0$ if $d_{g}(x_{0},x)  <  \frac{\iota_{g}}{40R}$. We define
\begin{equation}\label{def:vtr}
v_{t,R}^{\sigma}(x) := {\eta}_{R}(x) ~ u_{t}^{\sigma}(x)\hbox{ for all }x\in M.
\end{equation} 
Then we have 
\begin{proposition}\label{prop:v}
\begin{align}
\lim \limits_{R \rightarrow + \infty}  v_{t,R}^{\sigma}= u_{t}^{\sigma} ~ \text{ in } H_{k}^{2}(M) ~ \text{ uniformly }\forall (\sigma,t) \in S^{n-1} \times [0,\iota_{g}/2).
\end{align}
\end{proposition}
\smallskip\noindent{\it Proof of Proposition \ref{prop:v}:} We first note  that for all $(\sigma,t) \in S^{n-1} \times [0,\iota_{g}/2)$ the functions $u_{t}^{\sigma}$ are uniformly bounded in $C^{2k}$-norm in the ball $B_{x_{0}}(\frac{\iota_{g}}{20}) \subset M$.  And for any nonnegative integer $\alpha$,  one has $|\nabla_{g}^{\alpha} \eta_{R} |_{g}  \leq C R^{ \alpha}$. Therefore 
\begin{eqnarray}
&&\left\| v_{t,R}^{\sigma}-  u_{t}^{\sigma} \right\|_{H_{k}^2}^2   = 
 \sum \limits_{\alpha =0}^{k} \int \limits_{M} ( \Delta_g^{{\alpha}/2} ( v_{t,R}^{\sigma}-  u_{t}^{\sigma}) )^2 ~{dv_g} \\
 &&= \sum \limits_{\alpha =0}^{k} \int \limits_{B_{x_{0}}(\frac{\iota_{g}}{20R})} ( \Delta_g^{{\alpha}/2} ( v_{t,R}^{\sigma}-  u_{t}^{\sigma}) )^2 ~{dv_g}\\
 &&  = \sum \limits_{\alpha =0}^{k} \int \limits_{B_{x_{0}}(\frac{\iota_{g}}{20R})} ( \Delta_g^{{\alpha}/2}( ( \eta_{R}- 1)  u_{t}^{\sigma}))^2 ~{dv_g} =  O\left( \frac{1}{R}  \right) \qquad \text{( as $n \geq 2k + 1$ )}
\end{eqnarray}
The  above convergence is uniform $w.r.t$   $(\sigma,t) \in S^{n-1} \times [0,\iota_{g}/2)$. This proves Proposition \ref{prop:v}.\hfill$\Box$

\medskip\noindent
So   it follows that, there exists $R_{0}>0$, large, such that for all $R \geq R_{0}$ one has
\begin{align}
J_P( v_{t,R}^{\sigma})  \leq J_P( u_{t}^{\sigma})  + 2 \theta < 2^{2k/n} \frac{ 1 }{K_{0}(n,k)} \qquad \forall~ (\sigma,t) \in S^{n-1} \times [0,\iota_{g}/2).
\end{align}
As one checks for any $(\sigma,t) \in S^{n-1} \times [0,\iota_{g}/2)$,  the functions $ v_{t,R_{0}}^{\sigma} \neq 0$, and has support in $ M \backslash B_{x_{0}} ({\iota_{g}}/{40R_{0}})$. Let $\epsilon_{0}>0$ be such that $M \backslash B_{x_{0}} ({\iota_{g}}/{40R_{0}})\subset M \backslash B_{x_{0}} (\epsilon_{0})$  
and we define 
\begin{equation}\label{def:omega}
\Omega_{\epsilon_{0}}:=M \backslash \overline{B_{x_{0}} (\epsilon_{0})}.
\end{equation}
Then for  any $(\sigma,t) \in S^{n-1} \times [0,\iota_{g}/2)$ the functions $v_{t,R_{0}}^{\sigma} \in H_{k,0}^{2}(\Omega_{\epsilon_{0}})\backslash \{0 \}$. 
Propositions \ref{bubble energy} and \ref{prop:v} yield
\begin{align}
\lim \limits_{t \rightarrow \iota_{g}/2} J_P( v_{t,R_{0}}^{\sigma})= \frac{ 1 }{K_{0}(n,k)} \qquad\hbox{ uniformly for all } \sigma \in S^{n-1}. 
\end{align}
\medskip\noindent
Also $v_{0,R_{0}}^{\sigma} $ is a fixed  function independent of $\sigma$ and 
\begin{align}
v_{t,R_{0}}^{\sigma} \rightharpoonup \delta_{ exp_{x_{0}}( \frac{\iota_{g}}{2}  \sigma)} \qquad \text{weakly in the sense of measures as }  t \rightarrow \iota_{g}/2.
\end{align}
 \medskip


\noindent We define $S_{k}:= K_{0}(n,k)^{-1}$. For any $c\in\mathbb{R}$, we define the sublevel sets of the functional $I_P$ on $\mathcal{N}_{\epsilon_{0}}$
\begin{align}
{\mathcal I}_{c}:= \{ u \in \mathcal{N}_{\epsilon_{0}}: I_P(u) < c  \}.
\end{align}
where $\mathcal{N}_{\epsilon_{0}} :=\{u\in H_{k,0}^2(\Omega_{\epsilon_0})/\, \Vert u\Vert_{\crit}=1\}$.
\begin{proposition}\label{prop:center of mass}
Suppose  $I_P(u) >  \frac{1}{K_{0}(n,k)} $ for all $u \in \mathcal{N}_{\epsilon_{0}}  $, then there exists $\sigma_{0} >0$ for which  there exists a continuous map
\begin{align}
\Gamma: {\mathcal I}_{S_{k} + \sigma_{0}} \longrightarrow {\overline{\Omega}}_{\epsilon_{0}} 
\end{align}
such that if $(u_{i}) \in  {\mathcal I}_{S_{k} + \sigma_{0}} $ is a sequence such that $| u_{i} |^{  \crit} ~dv_{g} \rightharpoonup  \delta_{p_{0}}$ weakly in the sense of measures, for some point $p_{0} \in {\overline{\Omega}}_{\epsilon_{0}} $, then 
\begin{align}
\lim \limits_{i \rightarrow + \infty}\Gamma(u_{i})= p_0.
\end{align}
\end{proposition}
\medskip\noindent{\it Proof of Proposition \ref{prop:center of mass}:} By the Whitney embedding theorem,  the manifold $M$ admits a smooth embedding into $\R^{2n+1}$.  If we denote this embedding  by $\mathcal{F}: M \rightarrow \R^{2n+1}$, then $M$ is diffeomorphic to $\mathcal{F}(M)$  where $\mathcal{F}(M)$ is an embedded submanifold of $\R^{2n+1}$. For $u \in \mathcal{N}_{\epsilon_{0}}$, we define
\begin{align}
\tilde{\Gamma}(u):=  \int \limits_{\Omega_{M}} \mathcal{F}(x)~ |u(x)|^{2_{k}^{\sharp}}~ dv_{g}(x).
\end{align}
Then  $\tilde{\Gamma}: \mathcal{N}_{\epsilon_{0}} \rightarrow \R^{2n+1} $ is continuous. Next we claim that for every $\epsilon >0 $ there exists a $\sigma >0$ such that
\begin{align}
u \in  {\mathcal I}_{S_{k} + \sigma} ~\Rightarrow ~ dist\left( \tilde{\Gamma}(u), \mathcal{F}({\overline{\Omega}}_{\epsilon_{0}} )\right) < \epsilon.
\end{align}
\medskip

\noindent
Suppose that the claim is not true, then there  exists an ${\epsilon}^{\prime} >0$ and  a sequence $(u_{i}) \in \mathcal{N}_{\epsilon_{0}}$, such that $\lim \limits_{i \rightarrow + \infty} I_P(u_{i}) =  S_{k}$  and $dist\left( \tilde{\Gamma}(u), \mathcal{F}({\overline{\Omega}}_{\epsilon_{0}} )\right)   \geq {\epsilon}^{\prime}$. Since there is no minimizer for $I_P$ on $\mathcal{N}_{\epsilon_{0}}$, it follows from Lemma \ref{lem:min} that for such a sequence $(u_{i})$ there exists a point $p_{0} \in \overline{\Omega}_{\epsilon_{0}}$ such that
$ | u_{i} |^{  \crit} dv_{g} \rightharpoonup  \delta_{p_{0}} ~ \text{weakly in the sense of measures}$. So $\tilde{\Gamma}(u_{i}) \rightarrow \mathcal{F}(p_{0})$, a contradiction since  $dist\left( \tilde{\Gamma}(u), \mathcal{F}({\overline{\Omega}}_{\epsilon_{0}} )\right)  \geq {\epsilon}^{\prime}$. This proves our claim. 
\medskip

\noindent
By the Tubular Neighbourhood Theorem, the embedded  submanifold $\mathcal{F}(M)$ has a tubular neighbourhood $\mathcal{U}$ in $\R^{2n+1}$ and there exists a smooth retraction
\begin{align}
\pi: \mathcal{U} \longrightarrow \mathcal{F}(M).
\end{align} 
Choose an $\epsilon_{0} >0$ small so that $\{ y \in \R^{2n+1} : dist\left( y, \mathcal{F}(M)\right) < \epsilon_{0} \} \subset \mathcal{U}$ . Then from our previous  claim it follows that, there exists $\sigma_{0}>0$ such that 
\begin{align}
 u \in    {\mathcal I}_{S_{k} + \sigma_{0}} ~\Rightarrow ~ \tilde{\Gamma}(u) \in  \mathcal{U}.
\end{align}
We define 
\begin{align}
{\Gamma}_M(u)= \mathcal{F}^{-1} \circ \pi \left(~ \int \limits_{M} \mathcal{F}(x)~ |u(x)|^{2_{k}^{\sharp}}~ dv_{g}(x)  \right). 
\end{align}
Then  the map $\Gamma_M:  {\mathcal I}_{S_{k} + \sigma_{0}} \to M$ is continuous.  Similarly as in our previous claim  we have:    for every $\epsilon >0 $ small  there exists $\delta >0$ such that
\begin{align}
u \in  {\mathcal I}_{S_{k}+ \delta} ~\Rightarrow ~d_{g}\left( {\Gamma}_{M}(u), {\overline{\Omega}}_{\epsilon_{0}} \right) < \epsilon.
\end{align}

\medskip\noindent Let $\pi^{{\overline{\Omega}}_{\epsilon_{0}}}: M \backslash B_{x_{0}}(\epsilon_{0}/2) \longrightarrow {\overline{\Omega}}_{\epsilon_{0}}$be a  retraction. Choose an $\epsilon^{\prime} >0$ small so that $\{ p \in M : d_{g}\left( p, {\overline{\Omega}}_{\epsilon_{0}}\right) < \epsilon^{\prime} \} \subset M \backslash B_{x_{0}}(\epsilon_{0}/2) $ . Then from our claim it follows that there exists a $\delta_{0}>0$ such that $ {\Gamma}_{M}(u) \in M \backslash B_{x_{0}}(\epsilon_{0}/2)$ for all $u \in {\mathcal I}_{S_{k} + \delta_{0}}$. So for $u \in {\mathcal I}_{S_{k} + \delta_{0}}$  we define ${\Gamma}(u):= \pi^{{\overline{\Omega}}_{\epsilon_{0}}} \circ {\Gamma}_{M}(u)$. Then the map $\Gamma$ satisfies the hypothesis of the proposition. This proves Proposition \ref{prop:center of mass}. This proposition is in the spirit of Proposition 4.4 of \cite{malchiodi}\hfill$\Box$

\medskip\noindent
Now we proceed to prove the first part of Theorem \ref{main}.  By the regularity result obtained in  Theorem \ref{main regularity},  it is sufficient  to show the existence of a non-trivial $H^{2}_{k,0}(\Omega_{\epsilon_{0}})$    weak solution to the equation  (see \eqref{weak:sol} for the definition)
 \begin{eqnarray}\label{eq:sol}
 \left \{ \begin{array} {lc}
          Pu  = \left| u \right|^{ \crit - 2} u \qquad  \text{in} ~ \Omega_{M}\\
          D^{\alpha} u=0  \ \  \  \qquad \qquad\text{on } ~\partial \Omega_{M} \quad \text{for} \ \ |\alpha| \leq k-1
            \end{array} \right. 
\end{eqnarray}
Suppose on the contrary the above equation only admits trivial solutions, we will show that this leads to a  contradiction. 


\medskip\noindent
Now suppose that  the functional $ {I}_P$ has no critical point in  $\mathcal{N}_{{\epsilon_{0}}}$, that is there is not weak solution to \eqref{eq:sol}. This is equivalent to the assertion that the functional 
\begin{equation}\label{def:F:P}
F_{P}(u) =  \frac{1}{2}\int \limits_{{{\Omega}}_{\epsilon_{0}}} u  P(u)~ dv_{g}-\frac{1}{{  \crit}}~ \int \limits_{{{\Omega}}_{\epsilon_{0}}} |u|^{  \crit}~ dv_{g}. 
\end{equation}
does not admit a nontrivial  critical point in $H_{k,0}^{2}({{\Omega}}_{\epsilon_{0}} )$. 
\begin{proposition}\label{prop:ps:cond}
If equation $(\ref{eq:sol})$ admits only the  trivial solution $u \equiv 0$, then the functional $I_P$ satisfies the $(P.S)_{c}$ condition  for $c \in (S_{k},  2^{\frac{2 k}{n} }S_{k})$. 
\end{proposition}
\noindent{\it Proof of Proposition \ref{prop:ps:cond}:} Let $(v_{i})  \in \mathcal{N}_{\epsilon_{0}}  $ be a Palais-Smale sequence for the functional $I_P$ such that $\lim \limits_{i \rightarrow + \infty} I_P(v_{i}) = c \in (S_{k},  2^{\frac{2 k}{n} }S_{k})$, if this exists. Define $u_{i}:= (I_P(v_{i}))^{\frac{1}{2_{k}^{\sharp}-2}}v_{i}$. Then $(u_{i})$  is  a Palais-Smale sequence for the functional $F_P$ on the space $ H_{k,0}^2({{\Omega}}_{\epsilon_{0}} )$ such that $\lim \limits_{i \rightarrow + \infty} F_{P}(u_{i})  \in (\frac{k}{n} S_{k}^{n/2k},  \frac{2 k}{n} S_{k}^{n/2k} )$. Since there is no nontrivial solution to \eqref{eq:sol}, it follows from the Struwe-decomposition for polyharmonic operators by the author \cite{mazumdar2} that there exists  $ d \in \N$  non-trivial functions $u^j \in \mathscr{D}^{k,2}(\R^n)$,  $j= 1,\ldots,d$, such that upto a subsequence the following holds
\begin{align}
\displaystyle  F_P(u_i)=   \sum \limits_{j=1}^d ~E(u^j) + o(1)  \qquad \text{ as $i \rightarrow + \infty$}
\end{align}
where $E(u):=  \frac{1}{2}\int \limits_{\R^n} (\Delta^{k/2} u)^2 {dx}  
 - \frac{1}{\crit} \int \limits_{\R^n} | u |^{ \crit} {dx}$. The $u^j$'s are nontrivial solutions in $\mathscr{D}^{k,2}(\R^n)$ to $ \Delta^{k} u = |u|^{ \crit-2} u$ on $\R^n$ or on $\{x\in \R^n/\, x_1<0\}$ with Dirichlet boundary condition (we refer to \cite{mazumdar2} for details). It then follows from Lemma 3 and 5 of Ge-Wei-Zhou \cite{gwz} that for any $j$, either $u^j$ has fixed sign and $E(u) = \frac{k}{n} S_k^{n/2k}$, or $u^j$ changes sign and $E(u) \geq \frac{2k}{n} S_k^{n/2k}$, contradicting $\lim \limits_{i \rightarrow + \infty} F_{P}(u_{i})  \in (\frac{k}{n} S_{k}^{n/2k},  \frac{2 k}{n} S_{k}^{n/2k} )$. Therefore the Palais-Smale condition holds at level $c \in (S_{k},  2^{\frac{2 k}{n} }S_{k})$. More precisely, there is even no Palais-Smale sequence at this level. This ends the proof of Proposition \ref{prop:ps:cond}.\hfill$\Box$




\medskip\noindent {\it Proof of Theorem \ref{main}:} By the Deformation Lemma (see Theorem II.3.11 and Remark II.3.12 in the monograph by Struwe \cite{struwe}), there exists an retraction $\beta: {\mathcal I}_{S_{k}+4\theta } \longrightarrow I^{k}_{S_{k}+ \sigma_{0}}$, where $\sigma_{0}$ is  as  given in Proposition \ref{prop:center of mass}. Let $r_{{\mathcal{N}}_{\epsilon_{0}}} : H_{k,0}^{2}({{\Omega}}_{\epsilon_{0}}) \backslash\{0\}\rightarrow \mathcal{N}_{\epsilon_{0}}$ be the projection given by $u \mapsto \frac{u}{\left\| u\right\|_{L^{\crit}} }$. Consider the map $h:S^{n-1} \times [0,\iota_{g}/2] \rightarrow {\overline{\Omega}}_{\epsilon_{0}} $ given by 
\begin{eqnarray}
h(\sigma,t) := \left \{ \begin{array} {lc}
             \Gamma  \circ  \beta   (r_{{\mathcal{N}}_{\epsilon_{0}}}(v_{t,R_{0}}^{\sigma}) ) \quad \text{ for   } \quad \qquad   t < \iota_{g}/2   \\
              \sigma_{ \iota_{g} / 2  }^{M} \qquad  \qquad  \  \quad   \quad   \ \text{ for }  \quad  \qquad  t= \iota_{g}/2
          \end{array} \right.
\end{eqnarray}
where $\sigma_{t}^{M} :=  exp_{x_{0}}({t}\sigma)$. This map is well defined and continuous by  Proposition \ref{prop:center of mass} and there exists $p_0\in {\overline{\Omega}}_{\epsilon_{0}}$ such that  
\begin{eqnarray}
h(\sigma,t) = \left \{ \begin{array} {ll}
             p_0 & \text{ for   }  \qquad   t =0 \\
            exp_{x_{0}}( \frac{\iota_{g} } {2} \sigma   )&  \text{ for }   \qquad  t = \iota_{g}/2 
          \end{array} \right.
\end{eqnarray} 
\medskip

\noindent
So we obtain a homotopy of the embedded $(n-1)-$ dimensional sphere $\{    exp_{x_{0}}( \frac{\iota_{g} } {2} \sigma   ) : \sigma \in S^{n-1}\}$to a point in 
${{\Omega}}_{\epsilon_{0}} $,  which is a contradiction to our topological  assumption. This proves Theorem \ref{main} for potentially sign-changing solutions. 
 
 \section{Positive solutions}\label{sec:pos}
This section is devoted to the second part of Theorem \ref{main}, that is the existence of positive solutions. The proof is very similar to the proof of Theorem \ref{main} with no restriction on the sign. We just stress on the specificities and refer to the proof of Theorem \ref{main} everytime it is possible.
\smallskip\noindent We let $\Omega_M\subset M$ be any smooth $n-$dimensional submanifold of $M$, possibly with boundary. In the sequel, we will either take $\Omega_M=M$, or $M\setminus \overline{B_{x_0}(\epsilon_0)}$. For $u \in H_{k,0}^{2}(\Omega_{M})$, we define $u^{+}:= \max \{u,0 \}$, $u^{-}:= \max \{-u,0 \}$ and 
\begin{align}
 \mathcal{N}_{+} := \{  u \in H^{2}_{k,0}(\Omega_{M}):  \int \limits_{\Omega_{M}} (u^{+})^{\crit} ~dv_{g} = 1\}.
 \end{align}
which is a codimension 1 submanifold of $H_{k,0}^{2}(\Omega_{M})$. Any critical point $u\in H_{k,0}^{2}(\Omega_{M})$ of $I_g$ on $\mathcal{N}_+$ is a weak solution to
\begin{equation}\label{eq:u:plus}
Pu = u_+^{\crit-1}\hbox{ in }\Omega_M\; ; \, D^\alpha u=0\hbox{ on }\partial \Omega_M\hbox{ for }|\alpha|\leq k-1.
\end{equation} 
Consider the Green's function $G_{P}$  associated to  the operator $P$   with Dirichlet boundary condition on the smooth domain $\Omega_{M}\subsetneq M$, which is a function $G_{P} : \Omega_{M} \times \Omega_{M} \backslash \{(x,x) : x \in \Omega_{M}\} \longrightarrow  \R$ such that
\begin{enumerate}
\item[(i)] For any $x \in \Omega_{M}$,  the function $G_{P}(x,\cdot) \in L^{1}(\Omega_{M})$\\
\item[(ii)] For any $\varphi \in C^\infty(\overline{\Omega_{M}})$ such that $D^\alpha \varphi=0$ on $\partial\Omega_M$ for all $|\alpha|\leq k-1$, we have that 
\begin{align}
{\varphi}(x)= \int \limits_{\Omega_{M}} G_{P}(x,y) ~  P\varphi(y)~ dv_{g}(y).
\end{align}
\end{enumerate}

\begin{lemma}\label{lem:pos} Let $(u_{i}) \in \mathcal{N}_{+} $ be a minimizing sequence for $I^{k}_g$ on $\mathcal{N}_+$. Then
\begin{itemize}
\item[(i)] Either there exists $u_0\in \mathcal{N}_+$ such that $u_i\to u_0$ strongly in $H_{k,0}^{2}(\Omega_{M})$, and $u_0$ is a minimizer of $I_P$ on $\mathcal{N}_+$
\item[(ii)] Or there exists $x_0\in \overline{\Omega_M}$ such that $|u_i|^{\crit}\, dv_g\rightharpoonup \delta_{x_0}$ as $i\to +\infty$ in the sense of measures. Moreover, $\inf \limits_{u\in \mathcal{N}_{+} } I_P(u)=\frac{1}{K_0(n,k)}$.
\end{itemize}
\end{lemma}
\smallskip\noindent{\it Proof of Lemma \ref{lem:pos}:} As the functional $I_{g}$ is coercive so the sequence $(u_{i})$ is bounded in  $H^{2}_{k,0}(\Omega_{M})$. We let $u_0\in H^{2}_{k,0}(\Omega_{M})$ such that, up to a subsequence, $u_{i} \rightharpoonup u_{0}$ weakly in $H^{2}_{k,0}(\Omega_{M})$ as $i\to +\infty$, and $u_i(x)\to u_0(x)$ as $i\to +\infty$ for a.e. $x\in \Omega_M$. As the sequences $(u^{+}_{i}),(u^{-}_{i}) $ is bounded in $L^{2^{\sharp}_{k} } (\Omega_{M}) $ and $u^{+}_{i}(x) \rightarrow u^{+}_{0}(x)$, $u^{-}_{i}(x) \rightarrow u^{-}_{0}(x)$ for a.e. $x \in \Omega_{M}$, integration theory yields 
\begin{equation}
u_{i}^{+} \rightharpoonup u_{0}^{+} \hbox{ and }u_{i}^{-} \rightharpoonup u_{0}^{-}\qquad \text{weakly in } L^{2^{\sharp}_{k} } (\Omega_{M})\hbox{ as }i\to +\infty. 
\end{equation}
Therefore,
\begin{align}\label{bnd:u0}
 \left\|u^{+}_0\right\|_{L^{2_{k}^{\sharp}}}^{2^{\sharp}_{k}} \leq  \liminf_{i\to +\infty}\left\|u^{+}_i\right\|_{L^{2_{k}^{\sharp}}}^{2^{\sharp}_{k}}=1 \qquad \text{and }   \qquad \left\|u^{-}_0\right\|_{L^{2_{k}^{\sharp}}}^{2^{\sharp}_{k}} \leq   \liminf \limits_{\i \rightarrow + \infty}  \left\|u^{-}_{i}\right\|_{L^{2_{k}^{\sharp}}}^{2^{\sharp}_{k}}. 
\end{align}
\medskip\noindent We claim that 
\begin{equation}\label{lim:u:minus}
u_{i}^{-}\to  u_{0}^{-} \qquad \text{strongly in } L^{2^{\sharp}_{k}}(\Omega_{M}) .
\end{equation}
We prove the claim. We define $v_{i}:=u_{i}-u_{0}$. Up to extracting a subsequence, we have that $(v_i)_i\to 0$ in $H_{k-1}^2(M)$. Therefore, as $i \rightarrow + \infty$,
\begin{align}\label{ineq:vi:2}
I_P(u_i) =   \int \limits_{\Omega_{M}}  (\Delta_g^{k/2} v_{i}  )^{2} ~{dv_g}+  I_P(u_0)+ o(1).
\end{align}
And then, letting $\alpha:=\inf \limits_{u\in \mathcal{N}_{+} } I_P(u)$, we have that
$$ \alpha =I_P(u_i) +o(1)\geq  \int \limits_{\Omega_{M}}  (\Delta_g^{k/2} v_{i} )^{2} ~{dv_g} +   \alpha   \left\|u^{+}_0\right\|_{L^{2_{k}^{\sharp}}}^{2} +o(1)$$
and then
\begin{equation}\label{ineq:vi}
\alpha \left(  1-   \left\|u^{+}_0\right\|_{L^{2_{k}^{\sharp}}}^{2} \right) \geq  \int \limits_{\Omega_{M}}  (\Delta_g^{k/2} v_{i})^{2} ~{dv_g} + o(1)
\end{equation}
as $i\to +\infty$. \medskip
We fix $\epsilon>0$. It then follows from \eqref{eq:eps} and $(v_i)_i\to 0$ in $H_{k-1}^2(M)$ that 
\begin{align}
 \alpha \left( K_{0}(n,k) + \epsilon \right) \left(  1-   \left\|u^{+}_0\right\|_{L^{2_{k}^{\sharp}}}^{2} \right) \geq    \left\|v_{i}\right\|_{L^{2_{k}^{\sharp}}}^{2}  + o(1).
\end{align}
Since $1-a^{2^{\sharp}_{k}/2} \geq (1-a^{2})^{ 2^{\sharp}_{k}/2}$  for  $1 \geq a \geq 0$, we get that
\begin{align}
  \left(  \alpha  \left( K_{0}(n,k) + \epsilon \right)  \right)^{2^{\sharp}_{k}/2} \left( 1-   \left\|u^{+}_0\right\|_{L^{2_{k}^{\sharp}}}^{2^{\sharp}_{k}}  \right)   \geq    \left\|v_{i}\right\|_{L^{2_{k}^{\sharp}}}^{2^{\sharp}_{k}}   + o(1).
\end{align}
Integration theory yields $\left\|u_i \right\|_{L^{2_{k}^{\sharp}}}^{2_{k}^{\sharp}}= \left\|v_i \right\|_{L^{2_{k}^{\sharp}}}^{2_{k}^{\sharp}} + \left\|u_0 \right\|_{L^{2_{k}^{\sharp}}}^{2_{k}^{\sharp}} + o(1)$ as $i\to +\infty$. Therefore
\begin{eqnarray*}
&&  \left( \alpha \left( K_{0}(n,k) + \epsilon \right)  \right)^{2^{\sharp}_{k}/2} \left( 1-   \left\|u^{+}_0\right\|_{L^{2_{k}^{\sharp}}}^{2^{\sharp}_{k}}  \right)     + o(1)   \geq    \left\|u_i \right\|_{L^{2_{k}^{\sharp}}}^{2_{k}^{\sharp}} - \left\|u_0 \right\|_{L^{2_{k}^{\sharp}}}^{2_{k}^{\sharp}} \notag\\
  && =  \left\|u_i^{+} \right\|_{\crit}^{2_{k}^{\sharp}} - \left\|u_0^{+} \right\|_{\crit}^{2_{k}^{\sharp}} +  \left\|u_i^{-} \right\|_{\crit}^{2_{k}^{\sharp}} - \left\|u_0^{-} \right\|_{\crit}^{2_{k}^{\sharp}} = 1- \left\|u_0^{+} \right\|_{\crit}^{2_{k}^{\sharp}} +  \left\|u_i^{-} \right\|_{\crit}^{2_{k}^{\sharp}} - \left\|u_0^{-} \right\|_{\crit}^{2_{k}^{\sharp}} 
  \end{eqnarray*}
Then $\left\|u_i^- \right\|_{L^{2_{k}^{\sharp}}}^{2_{k}^{\sharp}}= \left\|u_i^--u_0^- \right\|_{L^{2_{k}^{\sharp}}}^{2_{k}^{\sharp}} + \left\|u_0^- \right\|_{L^{2_{k}^{\sharp}}}^{2_{k}^{\sharp}} + o(1)$ as $i\to +\infty$ yields
  \begin{align}
 \left(  \left( \mu \left( K_{0}(n,k) + \epsilon \right)  \right)^{2^{\sharp}_{k}/2} -1\right) \left( 1-   \left\|u^{+}_0\right\|_{L^{2_{k}^{\sharp}}}^{2^{\sharp}_{k}}  \right)    + o(1) & \geq   \left\|u_i^{-} \right\|_{L^{2_{k}^{\sharp}}}^{2_{k}^{\sharp}} - \left\|u_0^{-} \right\|_{L^{2_{k}^{\sharp}}}^{2_{k}^{\sharp}} \\
 & =  \left\|u_{i}^{-}- u_0^{-} \right\|_{L^{2_{k}^{\sharp}}}^{2_{k}^{\sharp}}  + o(1).
\end{align}   
Since $\alpha K_0(n,k)\leq 1$ and $\epsilon>0$ is arbitrary small, we get \eqref{lim:u:minus}. This proves the claim.

\medskip\noindent We define $\mu_{i} := (\Delta_{g}^{k/2} u_{i})^{2}~dv_{g}$ and $\nu_{i} = | u_{i}|^{2_{k}^{\sharp}}~dv_{g}$ for all $i$. Up to a subsequence, we denote respectively by $\mu$ and $\nu$ their limits in the sense of measures. It follows from the concentration-compactness Theorem \ref{th:CC} that, 
\begin{equation}
\nu = |u_{0}|^{2_{k}^{\sharp}}\, dv_g + \sum \limits_{j \in \mathcal{J}} \alpha_{j} \delta_{x_{j}} \hbox{ and }
\mu \geq  (\Delta_{g}^{k/2} u_{0})^{2}\, dv_g + \sum \limits_{j \in \mathcal{I}} \beta_{j} \delta_{x_{i}}
\end{equation}
where $J\subset \mathbb{N}$ is at most countable,  $(x_{j})_{j\in J} \in M$ is a family of points, and $(\alpha_{j})_{j\in J}\in \mathbb{R}_{\geq 0}$, $(\beta_j)_{j\in J}\in \mathbb{R}_{\geq 0}$ are such that $\alpha_{j}^{2/{2_{k}^{\sharp}}} \leq K_{0}(n,k) ~ \beta_{j}$ for all $j\in J$. Since $u_i^-\to u_0^-$ strongly in $L^{\crit}(M)$, we then get that
\begin{equation}
| u_{i}^+|^{2_{k}^{\sharp}}~dv_{g}\rightharpoonup |u_{0}^+|^{2_{k}^{\sharp}}\, dv_g + \sum \limits_{j \in \mathcal{J}} \alpha_{j} \delta_{x_{j}} 
\end{equation}
as $i\to +\infty$ in the sense of measures. The sequel is similar to the proof of Lemma \ref{lem:min}. We omit the details. This completes the proof of Lemma \ref{lem:pos}.\hfill$\Box$

\begin{lemma}\label{lem:ps:pos} We assume that there is no nontrivial solution to \eqref{eq:u:plus}. Then the functional $I_P$ satisfies the $(P.S)_{c}$ condition on $\mathcal{N}_{+}$ for $c \in (S_{k},  2^{\frac{2 k}{n} }S_{k})$  if the  equation.
\end{lemma}

\smallskip\noindent{\it Proof of Lemma \ref{lem:ps:pos}:} This is equivalent to prove that the functional
\begin{align}
F_{P}^{+}(u) =  \frac{1}{2}\int \limits_{{{\Omega}}_{\epsilon_{0}}} u  P(u)~ dv_{g}-\frac{1}{{  2_k^{\sharp}}}~ \int \limits_{{{\Omega}}_{\epsilon_{0}}} (u^{+})^{  2_k^{\sharp}}~ dv_{g} .
\end{align}
satisfies the $(P.S)_{c}$ condition on $H_{k,0}^{2}({{\Omega}}_{\epsilon_{0}} )$ for $c \in (\frac{k}{n} S_{k}^{n/2k},  \frac{2 k}{n} S_{k}^{n/2k})$.  Let $(u_{i})$  be  a Palais-Smale sequence for the functional $F_P^{+}$ on the space $ H_{k,0}^2({{\Omega}}_{\epsilon_{0}} )$.  Then, as $v \in   H_{k,0}^2({{\Omega}}_{\epsilon_{0}} )$ goes to $0$,
\begin{equation}\label{ps:plus}
 \int \limits_{{{\Omega}}_{\epsilon_{0}}} u_{i}  P_{g}^{k}(v)~ dv_{g}- \int \limits_{{{\Omega}}_{\epsilon_{0}}} (u_{i}^{+})^{  2_k^{\sharp}-1} v ~ dv_{g}  = o \left( \| v\|_{H^{2}_{k} } \right).
\end{equation}
\noindent  Without loss of generality we can assume that  $u_{i} \in C^{\infty}_c(\Omega_{\epsilon_{0}})$ for all $i$. Let $\varphi_{i} \in C^{\infty}( \overline{\Omega}_{\epsilon_{0}})$ be the unique solution of the equation
 \begin{eqnarray}
 \left \{ \begin{array} {ll}
          P_{g}^{k}\varphi_{i}  = (u_{i}^{+})^{2^{\sharp}_{k}-1} &  \text{in} ~ \Omega_{{\epsilon_{0}}}\\
          D^{\alpha} \varphi_{i}=0    &\text{on } ~\partial \Omega_{{\epsilon_{0}}} \quad \text{for} \ \ |\alpha| \leq k-1.            \end{array} \right. 
\end{eqnarray}  
The existence  of such  $\varphi_{i}$ is guaranteed by Theorem \ref{existence and uniqueness}. It then follows from Green's representation formula that
\begin{align}
\varphi_{i}(x)= \int \limits_{\Omega_{M}} G_{P}(x,y) (u_{i}^{+}(y))^{2^{\sharp}_{k}-1}~ dv_{g}(y) \geq 0.
\end{align} 
for all $x\in \Omega_M$. Note that the sequence $(\varphi_{i})$ is bounded in $H_{k,0}^{2}({{\Omega}}_{\epsilon_{0}} )$. It follows from \eqref{ps:plus} that $\varphi_{i}= u_{i} + o(1)$, where $o(1) \to 0$ in $ H_{k,0}^2({{\Omega}}_{\epsilon_{0}} )$ as $ i \to + \infty$. And so $(\varphi_{i})$ is  Palais-Smale sequences for the functional $F_P^+$ on the space $ H_{k,0}^2({{\Omega}}_{\epsilon_{0}} )$. Therefore, since $\varphi_i\geq 0$, it is also a Palais-Smale sequence for $F_P$ defined in \eqref{def:F:P}. Since there is no nontrivial critical point for $F_P^+$, using the Struwe decomposition \cite{mazumdar2} as in the proof of Proposition \ref{prop:ps:cond}, we then get that $(\varphi)_i$ is relatively compact in  $ H_{k,0}^2({{\Omega}}_{\epsilon_{0}} )$, and so is $(u_i)$. This ends the proof of Lemma \ref{lem:pos}. \hfill$\Box$

\medskip\noindent {\it Proof of Theorem \ref{main}, positive solutions:} this goes essentially as in the proof of Theorem \ref{main}, the key remark being that the functions $v_{t,R}^{\sigma}$ defined in \eqref{def:vtr} are nonnegative. We define $\mathcal{N}^{\epsilon_{0}}_{+}= \{ u \in H_{k,0}^{2} (\Omega_{\epsilon_{0}}): \left\| u^{+}\right\|_{L^{2_k^{\sharp}}} =1 \}$, 
where $\Omega_{\epsilon_{0}}=M\setminus \overline{B}_{\epsilon_0}(x_0)$ and $\epsilon_0>0$ was defined in \eqref{def:omega}. For  $c \in \R$ we define the sublevel sets of the functional $I_P$ on $ \mathcal{N}^{\epsilon_{0}}_{+}$ as ${\mathcal I}_{c}^+:= \{ u \in \mathcal{N}^{\epsilon_{0}}_{+} : {I}_{g}^{k}(u)  < c  \}$. Arguing as in the proof of Proposition \ref{prop:center of mass}, it follows from Lemma \ref{lem:pos} that there exists a $\delta_{0} >0$ such that there exists $\Gamma: {\mathcal I}^+_{S_{k} + \delta_{0}} \to  {\overline{\Omega}}_{\epsilon_{0}} $ a continous map such that: If $(u_{i}) \in {\mathcal I}^+_{S_{k} + \delta_{0}} $ is a sequence such that $| u_{i}^+ |^{  2_k^{\sharp}} ~dv_{g} \rightharpoonup  \delta_{p_{0}}$ weakly in the sense of measures, for some point $p_{0} \in {\overline{\Omega}}_{\epsilon_{0}} $, then $\lim \limits_{i \rightarrow + \infty}\Gamma(u_{i})= p_{0}$.
 
\medskip\noindent
Let $\displaystyle{r_{\mathcal{N}^{\epsilon_{0}}_{+}} : H_{k,0}^{2}({{\Omega}}_{\epsilon_{0}}) \backslash\{ \left\| u^{+}\right\|_{L^{2_k^{\sharp}}}=0\}\rightarrow \mathcal{N}^{\epsilon_{0}}_{+}}$ be the map given by $u \mapsto \frac{u}{\left\| u^{+}\right\|_{L^{2_k^{\sharp}}} }$. Consider the map $h:S^{n-1} \times [0,\iota_{g}/2] \rightarrow {\overline{\Omega}}_{\epsilon_{0}} $ given by 
\begin{eqnarray}
h(\sigma,t) = \left \{ \begin{array} {lc}
             \Gamma  \circ  \beta   (r_{\mathcal{N}^{\epsilon_{0}}_{+}}(v_{t,R_{0}}^{\sigma}) ) \quad \text{ for   } \quad \qquad   t < \iota_{g}/2   \\
              \sigma_{ \iota_{g} / 2  }^{M} \qquad  \qquad  \  \quad   \quad   \ \text{ for }  \quad  \qquad  t= \iota_{g}/2
          \end{array} \right.
\end{eqnarray}
where $\beta: {\mathcal I}^+_{S_k+4\theta}\to {\mathcal I}^+_{S_k+\delta_0}$ is a retract (we have used Lemma \ref{lem:ps:pos}) and  $\sigma_{t}^{M} =  exp_{x_{0}}({t}\sigma)$. Note here that we use that $v_{t,R_{0}}^{\sigma}\geq 0$. As in the proof of Theorem \ref{main}, $h$ is an homotopy of the embedded $(n-1)-$dimensional sphere $\{    exp_{x_{0}}( \frac{\iota_{g} } {2} \sigma   ) : \sigma \in S^{n-1}\}$ to a point in 
${{\Omega}}_{\epsilon_{0}} $,  which is a contradiction to our topological  assumption. So there  exists a nontrivial critical point $u$  for  the  functional $ I_P$ on   $\mathcal{N}^{\epsilon_{0}}_{+}$, which yields a weak solution to \eqref{eq:u:plus}. It then follows from the regularity theorem \ref{main regularity} that $u\in C^\infty(\overline{\Omega}_{\epsilon_0})$, $u>0$, is a solution to \eqref{eq:P:f}. This ends the proof of Theorem \ref{main} for positive solutions.\hfill$\Box$
  
\section{An Important  Remark}\label{sec:counterex}
 
 \noindent
We remark that the topological condition of Theorem \ref{main} is in general  a necessary condition. Consider the $n$-dimensional unit sphere  $\mathbb{S}^{n}$  endowed with its standard round metric $h_{r}$ and let  $P_{h_{r}}$ be the conformally invariant GJMS operator on $\mathbb{S}^{n}$.  By the stereographic projection it follows that   $\mathbb{S}^{n} \backslash \{ x_{0}\}$ is conformal to $\R^{n}$. Also one has that  $\mathbb{S}^{n}\backslash \{ x_{0}\}$  is contractible to a point.  Let $\Omega_{\epsilon_{0}}$ be the domain in $\mathbb{S}^{n} \backslash \{ x_{0}\}$ constructed as earlier in \eqref{main}, and let   $u\in H^{2}_{k,0} ( \Omega_{{\epsilon_{0}}})$, $u \neq 0$ solve the equation 
  \begin{eqnarray}
 \left \{ \begin{array} {lc}
          P_{h_{r}} u  = (u^{+})^{2_{k}^{\sharp}-1} \qquad  \text{in} ~ \Omega_{{\epsilon_{0}}}\\
          D^{\alpha} u=0  \ \  \  \qquad \qquad\text{on } ~\partial \Omega_{{\epsilon_{0}}} \quad \text{for} \ \ |\alpha| \leq k-1.            \end{array} \right. 
\end{eqnarray} 
 \medskip
 
 \noindent
 Then by  the stereographic projection it follows   that there exists a ball of radius $R$,  $B_{0}(R)$ such that there is a nontrivial solution  $v \in H^{2}_{k,0}(B_{0}(R))$ to the equation 
  \begin{eqnarray}{\label{lap k in ball}}
 \left \{ \begin{array} {lc}
          \Delta^{k}v  = (v^{+})^{2_{k}^{\sharp}-1} \qquad  \text{in} ~  B_{0}(R)\\
          D^{\alpha} v=0  \qquad \qquad \quad \text{on } ~\partial B_{0}(R) \quad \text{for} \ \ |\alpha| \leq k-1. \\
           \end{array} \right. 
\end{eqnarray} 
By a result of  Boggio\cite{boggio}, the Green's function for the Dirichlet problem above is positive. Therefore, we get that $v>0$ is a smooth classical solution to 
  \begin{eqnarray}
 \left \{ \begin{array} {lc}
          \Delta^{k}v  = v^{2_{k}^{\sharp}-1} \qquad  \text{in} ~  B_{0}(R)\\
          D^{\alpha} v=0  \qquad  \qquad \text{on } ~\partial B_{0}(R) \quad \text{for} \ \ |\alpha| \leq k-1. \\
           \end{array} \right. 
\end{eqnarray} 
This is impossible by Pohozaev identity, see Lemma 3 of Ge-Wei-Zhou \cite{gwz}. 

 \section{Appendix: Regularity}
 
 \noindent
 Let  $f \in L^{1}(\Omega_{M})$. We say that  $u \in H^{2}_{k,0}(\Omega_{M})$ is  a weak solution of the  equation $\displaystyle{  Pu  = f }$  in $\Omega_{M}$ and $D^{\alpha} u=0$ on $\partial \Omega_{M}$ for $|\alpha| \leq k-1$, if  for all $\varphi \in C_c^\infty(\Omega_{M})$ 
\begin{equation}
  \int \limits_{\Omega_{M}} \Delta_g^{k/2} u \, \Delta_g^{k/2} \varphi  ~{dv_g} + 
  \sum \limits_{\alpha =0}^{k-1}~ \int \limits_{\Omega_{M}} 
   A_{l}(g) (\nabla^{l} u,\nabla^{l} \varphi) ~dv_{g}
     = \int \limits_{\Omega_{M}}  f \varphi ~{dv_g}. \label{weak:sol}
\end{equation}

\medskip\noindent Now let the operator  $P$ be  coercive on the space $H_{k,0}^{2}(\Omega_M)$, i.e there exists  a constant  $C>0$ such that for all $u \in H^{2}_{k,0}(\Omega_M)$
\begin{equation}
\int \limits_{\Omega_{M}} u P(u)~ dv_{g} \geq \  C \left\|u \right\|_{H_{k,0}^2(\Omega_{M})}^2.
\end{equation}
\medskip
We then have
\begin{proposition}[$(H^{p}_{k}$-coercivity]\label{coercivity lemma}
\begin{align}
\inf \limits_{u \in  H_k^p(\Omega_{M}) \backslash \{0\}  } \frac{\left\| Pu\right\|_{p}}{\left\|u \right\|_{H^{p}_{k}} } >0.
\end{align}
\end{proposition}
\smallskip\noindent{\it Proof of Proposition \ref{coercivity lemma}:} We proceed by contradiction. If not, then there exists a sequence $(u_{i}) \in C_c^\infty(\Omega_{M})$ such that $\left\|u_{i} \right\|_{H^{p}_{k}} =1 $ and $\displaystyle{ \lim \limits_{i \rightarrow + \infty} \left\| Pu_{i}\right\|_{p}} =0$. It follows from classical estimates (see Agmon-Douglis-Nirenberg \cite{ADN}) that
\begin{align}
\left\| u_{i} \right\|_{H^{p}_{2k}(\Omega_{M})} &\leq C_{p} \left( \left\|  Pu_{i} \right\|_{L^{p}} + \left\| u_{i} \right\|_{H^{p}_{k}} \right) =O(1).
\end{align}
So  there exists $u_{0} \in H^{p}_{ 2k,0}(\Omega_{M}) $ such that upto a subsequence  $u_{i} \rightharpoonup u_{0} $ weakly in $H^{p}_{ 2k,0}(\Omega_{M}) $. Then     $u_{i} \rightarrow u_{0} $ strongly  in $H^{p}_{k, 0}(\Omega_{M}) $ and so $\left\|u_{0} \right\|_{H^{p}_{k}}=1$. Also $u_{0}$ weakly solves the equation $Pu_0  = 0$ in $\Omega_M$ and $D^{\alpha} u_0=0$ on $\partial \Omega_{M}$ for $|\alpha| \leq k-1$. It follows from standard elliptic estimates (see Agmon-Douglis-Nirenberg \cite{ADN}) that $u_{0} \in C^{\infty}(\overline{\Omega}_{M})$. Then, multiplying the equation by $u_0$ and integrating over $M$, coercivity yields
\begin{align}
 C \left\|u_{0} \right\|_{H_{k}^2(\Omega_{M})}^2 \leq \int \limits_{{M}} u_{0} Pu_{0}~ dv_{g}=0 .
\end{align}
and hence $u_{0} \equiv 0$, a contradiction since we have also obtained that  $\left\|u_{0} \right\|_{H^{p}_{k}}=1$. This proves Proposition \ref{coercivity lemma}.\hfill$\Box$
 
\begin{proposition}[Existence and Uniqueness] {\label{existence and uniqueness}} 
Let the operator $P$ be coercive. Then given any $f \in L^{p}(\Omega_{M})$, $1 < p < + \infty$,  there exists a  unique weak solution $u \in H^{p}_{k,0}(\Omega_{M}) \cap H^{p}_{2k} (\Omega_{M}) $ to 
\begin{eqnarray}
 \left \{ \begin{array} {lc}
          Pu  = f \qquad  \text{in} ~ \Omega_{M}\\
          D^{\alpha} u=0   \qquad\text{on } ~\partial \Omega_{M} \quad \text{for} \ \ |\alpha| \leq k-1  .          \end{array} \right. 
\end{eqnarray} 
\end{proposition}
The proof is classical and we only sketch it here. For $p=2$, existence and uniqueness follows from the Riesz representation theorem in Hilbert spaces. For arbitrary $p>1$, we approximate $f$ in $L^p$ by smooth compactly supported function on $\Omega_M$. For each of these smooth functions, there exists a solution to the pde with the approximation as a right-hand-side. The coercivity and the Agmon-Douglis-Nirenberg estimates yield convergence of these solutions to a solution of the original equation. Coercivity yields uniqueness.

\medskip

\noindent
We now proceed to prove our regularity results. The proof is based on ideas developed  by Van der Vorst \cite{vdv}, and also employed  by Djadli-Hebey-Ledoux \cite{DHL} for the case $k=2$. 

\begin{theorem}{\label{Lp}}
 Let $(M,g)$ be a  smooth, compact Riemannian manifold  of dimension $n$ and let $k$ be a positve integer such that $2k < {n}$.  Let $\Omega_{M}$ be a smooth domain in $M$ and suppose $u \in H^{2}_{k,0}(\Omega_{M})$  be a weak solution of the equation
 \begin{eqnarray}
 \left \{ \begin{array} {lc}
          Pu  = f(x,u) \qquad  \text{in} ~ \Omega_{M}\\
          D^{\alpha} u=0  \ \  \  \qquad \qquad\text{on } ~\partial \Omega_{M} \quad \text{for} \ \ |\alpha| \leq k-1
            \end{array} \right. 
\end{eqnarray}
where $ \left| f(x,u) \right| \leq C|u| (1+ \left| u \right|^{ \crit - 2} )$ for some positive constant  $C$, then
 \begin{align}
u \in L^{p}(\Omega_{M}) \qquad \text{ for all }~ 1< p < + \infty.
\end{align}
 
 \end{theorem}
\noindent{\it Proof of \ref{Lp}:} We write $f(x,u)= bu$ where $|b| \leq  C(1+ \left| u \right|^{ \crit - 2} )$. Then  $b \in  L^{n/2k} (\Omega_{M})$ and $u$ solves weakly the equation 
 \begin{eqnarray}\label{eq:P:b:u}
 \left \{ \begin{array} {lc}
          Pu  = b u \qquad  \text{in} ~ \Omega_{M}\\
          D^{\alpha} u=0  \ \ \qquad\text{on } ~\partial \Omega_{M} \quad \text{for} \ \ |\alpha| \leq k-1.
            \end{array} \right. 
\end{eqnarray}

\medskip\noindent{\bf Step 1:} We claim that for any $\epsilon >0$ there exists $q_{\epsilon} \in  L^{n/2k} (\Omega_{M}) $ and $f_{\epsilon} \in L^{\infty} ( \Omega_{M})$ such that
\begin{align}
bu= q_{\epsilon} u + f_{\epsilon} , ~ \qquad \left\| q_{\epsilon}  \right\|_{ L^{n/2k} (\Omega_{M})} < \epsilon.
\end{align}
\medskip

\noindent
Now  $\displaystyle{\lim \limits_{i \rightarrow + \infty}  \int   \limits_{\{  |u| \geq  i \}} |b|^{n/2k} ~ dv_{g} = 0}$,   so given any $\epsilon >0$ we can choose $i_{0} $  such that 
$$ \int   \limits_{\{  |u| \geq  i_{0} \}} |b|^{n/2k} ~ dv_{g} < \epsilon^{n/2k}.$$
We define $q_{\epsilon} := \chi_{\{  |u| \geq  i_{0} \}} b$ and $f_{\epsilon}:= (b-q_{\epsilon})u=\chi_{\{  |u| < i_{0} \}}b$.
Then, since $|b| \leq  C(1+ \left| u \right|^{ \crit - 2} )$, we have that $\left\| q_{\epsilon} \right\|_{L^{n/2k}(\Omega_{M})} <  \epsilon$ and $f_\epsilon\in L^\infty(M)$. This proves our claim and ends Step 1.

\smallskip\noindent We rewrite \eqref{eq:P:b:u} as
\begin{eqnarray}
 \left \{ \begin{array} {lc}
          Pu  = q_{\epsilon}u + f_{\epsilon} \qquad  \text{in} ~ \Omega_{M}\\
          D^{\alpha} u=0   \quad \qquad  \qquad\text{on } ~\partial \Omega_{M} \quad \text{for} \ \ |\alpha| \leq k-1.
            \end{array} \right. 
\end{eqnarray} 
Let $\mathscr{H}_{\epsilon}$ be the operator defined formally as 
\begin{align}
\mathscr{H}_{\epsilon} u = (P^{k}_{g})^{-1}(q_{\epsilon} u ).
\end{align}
Then $  Pu  = q_{\epsilon}u + f_{\epsilon} $ becomes $u  - \mathscr{H}_{\epsilon} u = (P^{k}_{g})^{-1}(f_{\epsilon}  )$.

\medskip\noindent{\bf Step 2:} we claim that for any $s>1$, $\mathscr{H}_{\epsilon}$ maps $L^s(\Omega_M)$ to $L^s(\Omega_M)$.

\smallskip\noindent We prove the claim. Let $v \in L^{s}(\Omega_{M})$, $s \geq 2^{\sharp}_{k}$, then   $q_{\epsilon} v \in L^{\hat{s}}(\Omega_{M})$ where $\displaystyle{\hat{s}=\frac{ns}{n+ 2ks} }$, and we have by H\"{o}lder inequality
\begin{align}
\left\| q_{\epsilon} v \right\|_{ L^{\hat{s}}(\Omega_{M})} \leq  \left\| q_{\epsilon}  \right\|_{ L^{n/2k}(\Omega_{M})} \left\| v \right\|_{ L^{s}(\Omega_{M})} 
\end{align}
Since $ \left\| q_{\epsilon}  \right\|_{ L^{n/2k} (\Omega_{M})} <  \epsilon$, so we have 
\begin{align}
\left\| q_{\epsilon} v \right\|_{ L^{\hat{s}}(\Omega_{M})} \leq  \epsilon \left\| v \right\|_{ L^{s}(\Omega_{M})} 
\end{align}
\medskip

\noindent
From {\eqref{existence and uniqueness}} it follows that there exists  a unique $v_{\epsilon} \in H^{\hat{s}}_{2k} (\Omega_{M})$ such that
\begin{eqnarray}
 \left \{ \begin{array} {lc}
          Pv_{\epsilon}  = q_{\epsilon} v  \qquad  \text{in} ~ \Omega_{M}\\
          D^{\alpha} v_{\epsilon}=0   \qquad\text{on } ~\partial \Omega_{M} \quad \text{for} \ \ |\alpha| \leq k-1            \end{array} \right. 
\end{eqnarray} 
weakly. Further we have for a positive constant $C(s)$
\begin{align}
\left\| v_{\epsilon} \right\|_{H^{\hat{s}}_{2k}(\Omega_{M})} \leq C(s)  \left\| q_{\epsilon} v \right\|_{ L^{\hat{s}}(\Omega_{M})}
\end{align}
So we obtained that
\begin{align}
\left\| v_{\epsilon}  \right\|_{H^{\hat{s}}_{2k}(\Omega_{M})} \leq C(s) \epsilon \left\| v \right\|_{ L^{s}(\Omega_{M})}. 
\end{align}
\medskip

\noindent
By Sobolev embedding theorem $H^{\hat{s}}_{2k} (\Omega_{M}) $ is continuously imbedded in $ L^{s}(\Omega_{M}) $  so $v_{\epsilon} \in L^{s}(\Omega_{M})$ and  we have 
\begin{align}
\left\| v_{\epsilon}  \right\|_ { L^{s}(\Omega_{M}) } \leq C(s) \epsilon \left\| v \right\|_{ L^{s}(\Omega_{M}) } .
\end{align}
\medskip

\noindent
In other words, for any  $s \geq 2^{\sharp}_{k}$ the operator $\mathscr{H}_{\epsilon}$ acts  from $ L^{s}(\Omega_{M}) $  into $  L^{s}(\Omega_{M})$, and its norm  $ \displaystyle{\left\| \mathscr{H}_{\epsilon} \right\|_{L^{s}  \rightarrow L^{s} }  \leq C(s) \epsilon } $. This proves the claim and ends Step 2.

\medskip\noindent {\bf Step 3:} Now let $s \geq 2^{\sharp}_{k}$ be given, then for $\epsilon >0$ sufficiently small one has
\begin{align}
\left\| \mathscr{H}_{\epsilon} \right\|_{L^{s}  \rightarrow L^{s} }  \leq \frac{1}{2}
\end{align}
and so the operator $I- \mathscr{H}_{\epsilon} : L^{s}(\Omega_{M}) \longrightarrow L^{s}(\Omega_{M}) $  is invertible. We have 
\begin{align}
u  - \mathscr{H}_{\epsilon} u = (P^{k}_{g})^{-1}(f_{\epsilon}  )
\end{align}
Since $u \in L^{2^{\sharp}_{k}}( \Omega_{M})$ and  $f_{\epsilon} \in L^{\infty} ( \Omega_{M})$, so $u \in L^{p}( \Omega_{M})$  for all $1< p < + \infty$.

\medskip\noindent This ends the proof of Theorem \ref{Lp}.\hfill$\Box$

\begin{proposition}{\label{main regularity}}
 Let $(M,g)$ be a  smooth, compact Riemannian manifold  of dimension $n$ and let $k$ be a positive integer such that $2k < {n}$.  Let $f\in C^{0,\theta}(\Omega_{M})$ a H\"older continuous function. Let $\Omega_{M}$ be a smooth domain in $M$ and suppose $u \in H^{2}_{k,0}(\Omega_{M})$ be a weak solution of the equation
 \begin{eqnarray}
 \left \{ \begin{array} {lc}
          Pu  = f \left| u \right|^{ \crit - 2} u \hbox{ or }f (u^{+})^{ \crit - 1}\qquad  \text{in} ~ \Omega_{M}\\
          D^{\alpha} u=0  \ \  \  \qquad \qquad\text{on } ~\partial \Omega_{M} \quad \text{for} \ \ |\alpha| \leq k-1.
            \end{array} \right. 
\end{eqnarray}
Then $u \in C^{2k}( {\Omega}_{M})$, and is a classical solution  of the above equation. Further if $u > 0$ and $f\in C^\infty(\Omega_M)$, then $u \in C^{\infty}( {{\Omega}}_{M})$. 
\end{proposition}
\smallskip\noindent{\it Proof of Proposition \ref{main regularity}:} It  follows from  \eqref{Lp} that $ u \in H^{p}_{2k}(\Omega_{M})$ for all $1 < p < + \infty$. By Sobolev imbedding theorem this implies  $u \in C^{2k-1, \gamma} (\overline{\Omega}_{M})$ for all $0 <\gamma< 1$.   $\left| u \right|^{ \crit - 2} u,(u^{+})^{ \crit - 1}  \in C^{1}(\overline{\Omega}_{M})$. The Schauder estimates (here again, we refer to Agmon-Douglis-Nirenberg \cite{ADN}) then yield $ u\in  C^{2k,\gamma}( \overline{\Omega}_{M}) $ for all $\gamma\in (0,1)$, and $u$ is  a classical solution. 

\smallskip\noindent If  $u > 0$, then the right-hand-side is $ u^{\crit-1}$ and has the same regularity as $u$. Therefore, iterating the Schauder estimates yields $u\in C^\infty(\overline{\Omega_M})$. This ends the proof of Proposition \ref{main regularity}.\hfill$\Box$



 \section{Appendix: Local Comparison of the Riemannian norm with the Euclidean norm}
 
 \noindent
 Let $(M,g)$ be a smooth, compact Riemannian manifold of dimension $n\geq 1$. For  any point $p\in M$ there exists a local coordinate around $p$,   $\varphi^{-1}_{p}: \Omega \subset \R^{n} \to M$, $\varphi(p)=0$,   such that in these   local coordinates  one has  for all indices $ i,j,k = 1,\ldots,n$ 
   $$\left\{ \begin{array}{ll}
 (1-\epsilon) \delta_{ij}\leq g_{ij}(x)\leq  (1+\epsilon)\delta_{ij} &~ \text{as bilinear forms.} \\
| g_{ij}(x) - \delta_{ij} | \leq \epsilon&\\
 \end{array}\right.$$
 
 \noindent
 Here we have identified $T_{p}M  \cong \R^{n}$ for any point $p \in M$.
 For example,  one can  take the exponetial map at $p$ : $exp_{p}$, which is normal at $p$.  We will let $\iota_g$ be  the injectivity radius of $M$. Using the above local comparison of the Riemannian metric with the Euclidean metric one obtains

\noindent

 \begin{lemma}{\label{normcomparison}}
 Let $(M,g)$ be  a  smooth, compact Riemannian manifold  of dimension $n$  and let $k$ be positive integer such that $2k <n$. We fix $s\geq 1$. Let  $\varphi^{-1}_{p}: \Omega \subset \R^{n} \to M$, $\varphi(p)=0$  be a local coordinate around $p$ with the above mentioned properties. 
 Then given any $\epsilon_{0} >0$  there exists   $ \tau\in (0,\iota_{g})$,  such that for any point $p \in M$, and  $u \in C_c^\infty\left( B_{0}(\tau)\right)$ one has
  \begin{align}
  (1-\epsilon_{0})\int \limits_{\R^n} ( \Delta^{k/2} u )^2 ~{dx}
  \leq \int \limits_{M}  ( \Delta^{k/2}_{{g}} (u \circ \varphi_{p} ) )^2  ~{dv_{g}}
  \leq  {(1+\epsilon_{0})} \int \limits_{\R^n} ( \Delta^{k/2} u )^2 ~{dx} 
  \end{align}
  and 
  \begin{align}
 (1-\epsilon_{0})\int \limits_{\R^n} |u|^{s}  ~ dx  \leq \int \limits_{M}  |u \circ \varphi_{p}  |^{s} ~ dv_{g}  \leq  {(1+\epsilon_{0})} \int \limits_{\R^n}  |u|^{s} ~ dx  
  \end{align}
 
 \end{lemma}
 
\smallskip\noindent{\it Proof of Lemma \ref{normcomparison}:} In terms of the   coordinate map $\varphi^{-1}_{p}: \Omega \subset \R^{n} \to M$,  for any $f \in C^2(M)$ we have 
\begin{equation}
 \lap_g f \left( x \right) = - g^{ij}(x)  \left( \frac{\partial^2 (f \circ \varphi^{-1} )}{\partial x_i \partial x_j}(x) -
 \Gamma_{ij}^k (\varphi(x)) \frac{\partial (f \circ \varphi^{-1} )}{\partial x_k}(x) \right).
\end{equation}
Since the manifold $M$ is compact, then given any $\epsilon >0$ there exists a  $\tau \in (0, \iota_g)$ depending only on  $(M,g)$, such that for any point $p \in M$ and for any $x \in B_{0}(\tau) \subset \R^{n} $ one has  for all indices $ i,j,k = 1,\ldots,n$ 
   $$\left\{ \begin{array}{ll}
 (1-\epsilon) \delta_{ij}\leq g_{ij}(x)\leq  (1+\epsilon)\delta_{ij} &~ \text{as bilinear forms.} \\
| g_{ij}(x) - \delta_{ij} | \leq \epsilon&\\
 \end{array}\right.$$
Without loss of generality we can assume that $\tau < 1$. We let $u \in C^\infty_c(\R^n)$ be such that $supp(u)\subset B_0(\tau )$. In the sequel, the constant $C$ will denote any positive constant depending only on $(M,g)$ and $\tau$: the same notation $C$ may apply to different constants from line to line, and even in the same line. All integrals below are taken over $B_0(\tau)$, and we will therefore omit to write the domain for the sake of clearness.


\medskip\noindent{\bf Case 1: $k$ is even.} We then write $k= 2 m $, $m\geq 1$. Then calculating in terms of local coordinates we obtain 
\begin{equation}
\left| \Delta_{g}^{m} (u \circ \varphi_{p}) (\varphi_{p}^{-1}(x))- \Delta^{m} u (x)  \right|\leq
 \epsilon \left|\nabla^{2 m} u (x)\right|  + C\sum \limits_{\beta=1}^{2 m -1} \left|\nabla^{(2m-\beta)} u (x)\right| 
\end{equation}
where $C_{g}$ is a constant depending only on  the metric $g$ on $M$. Then we have  
\begin{align}
&\left| ~ \int  \left( \Delta_{{g}}^{m}(u \circ \varphi_{p})(\varphi^{-1}_{p}(x))\right)^2~{dx}- \int  \left( \Delta^{m}u\right)^2 {dx}  \right| \leq 
 2^{2}  \epsilon^{2}     \int  \left|\nabla^{2 m} u \right|^2 {dx}  +  \notag\\
 & C    \sum \limits_{\beta=1}^{2 m -1} \int \left|\nabla^{(2m-\beta)} u \right|^2 {dx} 
+  2 \epsilon    \int  \left|\nabla^{2 m} u \right| \left| \Delta^{m}u\right| {dx} + C   \sum \limits_{\beta=1}^{2 m -1} \int  \left| \Delta^{m}u\right|\left|\nabla^{(2m- \beta)} u \right| {dx} .\\
\end{align}
Now   for any  $\beta$ such that $\beta \leq 2m-1$  we have $\nabla^{(2m-\beta)} u \in \mathscr{D}^{\beta ,2} (\R^n)$ and by Sobolev embedding theorem this implies that $\left|\nabla^{(2m- \beta)} u \right|^2  \in L^{2_{\beta}^{\sharp}/2}(\R^n)$. Applying the H\"{o}lder inequality we obtain 
\begin{equation}
 \sum \limits_{\beta=1}^{2 m -1} \int 
 \left|\nabla^{(2m-\beta)} u \right|^2 {dx}
  \leq C \left(  \sum \limits_{\beta=1}^{2 m -1}  \tau^{2 \beta}  \right) \left( \int 
    \left|\nabla^{(2m-\beta)} u \right|^{2_{\beta}^{\sharp}} {dx} \right)^{2/{2_{\beta}^{\sharp}}} 
\end{equation}
And then the Sobolev inequality gives us
\begin{equation}
\left( \int 
    \left|\nabla^{(2m-\beta)} u \right|^{2_{\beta}^{\sharp}} {dx} \right)^{2/{2_{\beta}^{\sharp}}} \leq
   C  \int   \left| \nabla^{2 m} u \right|^{2} {dx} .
\end{equation}
Applying the integration by parts formula, we obtain 
\begin{equation}
  \int  \left|\nabla^{2 m} u \right|^2 {dx} =  \int  
  \left(\Delta^{ m} u \right)^2 {dx}.
\end{equation}
\medskip

\noindent
So  we have, since $\tau <1$
\begin{align}
 \sum \limits_{|\beta|=1}^{2 m -1} \int 
 \left|\nabla^{2m-\beta} u \right|^2 {dx} \leq  C \tau  \int  
  \left(\Delta^{ m} u \right)^2 {dx}.
\end{align}
\medskip

\noindent
Therefore, we get that
\begin{align}
\left| ~ \int  \left( \Delta_{{g}}^{m}(u \circ \varphi_{p} )(\varphi^{-1}_{p}(x))\right)^2~{dx}- \int  \left( \Delta^{m}u\right)^2 {dx}  \right|  & \leq 
 C \left(  \epsilon     +  {\tau}   \right) \int   \left( \Delta^{m}u\right)^2 {dx} .
\end{align}
\medskip

\noindent
Now  in these  local coordinates  one has 
\begin{align}
&&(1- \epsilon)^{n/2} \int  \left( \Delta_{{g}}^{m}(u \circ \varphi_{p})(\varphi_{p}^{-1}(x))\right)^2~{dx}    \leq
 ~ \int \limits_{M} \left( \Delta_{{g}}^{m}(u \circ \varphi_{p})\right)^2~{dv_{g}}    \\
&&\leq (1+  \epsilon)^{n/2} ~ \int  \left( \Delta_{{g}}^{m}(u \circ \varphi_{p})(\varphi_{p}^{-1}(x))\right)^2~{dx} . 
\end{align}

\noindent
So given  an  $\epsilon_{0} >0$ small, we first choose  $\epsilon$ small and then choose  a sufficiently small  $\tau $, so that  for any  $u \in C_c^\infty\left( B_{0}(\tau)\right)$  we have

\begin{align}
\left| ~ \int  \left( \Delta_{{g}}^{m}(u \circ \varphi_{p})\right)^2~{dv_{{g}}}- \int  \left( \Delta^{m}u\right)^2 {dx}  \right|  \leq \epsilon_{0} \int  \left( \Delta^{m}u\right)^2 {dx}. 
\end{align}
So we   have the lemma for  $k$  even.

\medskip\noindent{\bf Case 2: $k$ is odd.} We then write $k= 2 m + 1$ with $m\geq 0$. Calculating in terms of local coordinates, like in the even case, we obtain
\begin{eqnarray}
&&\left|~  |\nabla( \Delta_{g}^{m} (u \circ \varphi_{p}))|^{2} (\varphi_{p}^{-1}(x))- |\nabla(\Delta^{m} u)| ^{2}(x)  \right| \leq  \epsilon |\nabla(\Delta^{m} u)| ^{2}(x)\\
&&+ C \epsilon \left|\nabla^{2 m+1} u \right|^{2} (x) + C\sum \limits_{\beta=1}^{2 m } \left|\nabla^{(2 m+1-\beta)} u \right|^{2} (x)  \\
&& + C\epsilon \left|\nabla^{2 m+1} u \right| (x) ~|\nabla(\Delta^{m} u)|(x)   + C\sum \limits_{\beta=1}^{2 m } \left|\nabla^{(2 m+1-\beta)} u \right| (x)~|\nabla(\Delta^{m} u)|(x)
\end{eqnarray}
for all $x\in B_0(\tau)$. Therefore 
\begin{align}
&\left| ~ \int  |\nabla( \Delta_{g}^{m} (u \circ \varphi_{p}))|^{2} (\varphi_{p}^{-1}(x)) ~{dx}- \int  |\nabla(\Delta^{m} u)| ^{2}(x)~{dx}  \right|  \leq  \epsilon \int   |\nabla(\Delta^{m} u)| ^{2} ~ dx \notag \\
&+ C \epsilon\int  \left|\nabla^{2 m+1} u \right|^{2}  ~ dx 
 + C\sum \limits_{\beta=1}^{2 m }\int  \left|\nabla^{(2 m+1-\beta)} u \right|^{2}  ~ dx   \notag\\
& + C \epsilon \int  \left|\nabla^{2 m+1} u \right|  |\nabla(\Delta^{m} u)|   ~ dx 
 + C\sum \limits_{\beta=1}^{2 m }\int  \left|\nabla^{(2 m+1-\beta)} u \right| |\nabla(\Delta^{m} u)|~ dx 
\end{align}
And then by  calculations similar to the even case, along with the integration by parts formula, we obtain 
\begin{align}
\left| ~ \int  |\nabla( \Delta_{g}^{m} (u \circ \varphi_{p}))|^{2} (\varphi_{p}^{-1}(x)) ~{dx}- \int  |\nabla(\Delta^{m} u)| ^{2}(x)~{dx}  \right| 
 \leq    \tilde{C}_{g} \left(  \epsilon + \sqrt{\tau}   \right)  \int   |\nabla(\Delta^{m} u)| ^{2} ~ dx.
\end{align}

\noindent
Now given  an  $\epsilon_{0} >0$ small, we first choose  $\epsilon$ small and then choose  a sufficiently small  $\tau $, so that  for any  $u \in C_c^\infty\left( B_{0}(\tau)\right)$  we have

\begin{align}
\left| ~ \int \limits_{M} |\nabla( \Delta_{g}^{m} (u \circ \varphi_{p}))|^{2}  ~{dv_{g}}- \int  |\nabla(\Delta^{m} u)| ^{2}~{dx}  \right|  \leq \epsilon_{0} \int  |\nabla(\Delta^{m} u)| ^{2}~{dx}.
\end{align}
\medskip

\noindent
Then one has the lemma for  $k$  odd. This ends the proof of Lemma \ref{normcomparison}.\hfill$\Box$

\end{document}